\documentclass[12pt]{amsart}
\usepackage{bm}

\usepackage{hyperref}
\usepackage[all]{xypic}
\usepackage{amsmath}
\usepackage{amscd}
\usepackage{amssymb}
\usepackage{enumerate}
\usepackage{amsfonts}
\usepackage{graphicx}
\usepackage[all]{xy}
\usepackage{tikz-cd}
\usepackage{mathrsfs}
\usepackage[active]{srcltx}
\usepackage{ftnxtra}

\DeclareFontEncoding{OT2}{}{}

\numberwithin{equation}{section}

\newcommand{\M}{\operatorname{M}}

\newcommand{\rar}{\longrightarrow}
\newcommand{\hrar}{\hookrightarrow}

\newcommand{\rarab}[1]{\overset{#1}{\longrightarrow}}

\newcommand{\al}{\alpha}
\newcommand{\be}{\beta}
\newcommand{\ga}{\gamma}
\newcommand{\Ga}{\Gamma}
\newcommand{\de}{\delta}

\newcommand{\la}{\lambda}
\newcommand{\La}{\Lambda}

\newcommand{\sg}{\sigma}

\newcommand{\om}{\omega}

\newcommand{\Sg}{\Sigma}

\newcommand{\bC}{{\mathbb C}}

\newcommand{\bN}{{\mathbb N}}

\newcommand{\bQ}{{\mathbb Q}}
\newcommand{\bR}{{\mathbb R}}

\newcommand{\bT}{{\mathbb T}}

\newcommand{\bZ}{{\mathbb Z}}

\newcommand{\cO}{{\mathcal O}}
\newcommand{\cP}{{\mathcal P}}

\newcommand{\caR}{{\mathcal R}}
\newcommand{\cS}{{\mathcal S}}

\newcommand{\cW}{{\mathcal W}}

\newcommand{\fA}{{\mathfrak A}}

\newcommand{\fp}{{\mathfrak p}}

\newcommand{\Span}{\operatorname{Span}}

\newcommand{\diag}{\operatorname{diag}}

\newcommand{\rank}{\operatorname{rank}}

\newcommand{\Gal}{\operatorname{Gal}}
\newcommand{\GL}{\operatorname{GL}}

\newcommand{\tr}{\operatorname{Tr}}

\newcommand{\sbr}{\smallbreak}

\newcommand{\indlim}[1]{\lim\limits_{\underset{N}{\longrightarrow}}}
\newtheorem{thm}{Theorem}[section]
\newtheorem{cor}[thm]{Corollary}
\newtheorem{lem}[thm]{Lemma}

\newtheorem{prop}[thm]{Proposition}
\newtheorem*{t:mainth0}{Theorem}

\theoremstyle{remark}
\newtheorem{rem}[thm]{Remark}

\newtheorem{defin}[thm]{Definition}

\newtheorem{example}{Example}

\newcommand{\bbe}{\begin{equation}}
\newcommand{\ee}{\end{equation}}

\title[A number theoretic classification of toroidal solenoids]{A number theoretic classification of toroidal solenoids}

\author{Maria Sabitova}
\date{\today}

\begin{document}

%\tableofcontents

\begin{abstract}  
We classify toroidal solenoids defined by non-singular $n\times n$-matrices $A$ with integer coefficients by studying associated first \^Cech cohomology groups. 
In a previous work, we classified the groups in the case $n=2$ using generalized ideal classes in the splitting field of the characteristic polynomial of $A$. In this paper we explore the classification problem for an arbitrary $n$.
 \end{abstract}

\maketitle

%%%%%%%%%%%%%%%%%%%%%%%%%%%%%%%%%%%%%%%%%%%%%%%%%
%%%%%%%%%%%%%%%%%%%%%%%%%%%%%%%%%%%%%%%%%%%%%%%%%
%%%%%%%%%%%%%%%%%%%%%%%%%%%%%%%%%%%%%%%%%%%%%%%%%
%%%%%%%%%%%%%%%%%%%%%%%%%%%%%%%%%%%%%%%%%%%%%%%%%
%%%%%%%%%%%%%%%%%%%%%%%%%%%%%%%%%%%%%%%%%%%%%%%%%
%%%
%%%            INTRODUCTION
%%%
%%%%%%%%%%%%%%%%%%%%%%%%%%%%%%%%%%%%%%%%%%%%%%%%%
%%%%%%%%%%%%%%%%%%%%%%%%%%%%%%%%%%%%%%%%%%%%%%%%%
%%%%%%%%%%%%%%%%%%%%%%%%%%%%%%%%%%%%%%%%%%%%%%%%%
%%%%%%%%%%%%%%%%%%%%%%%%%%%%%%%%%%%%%%%%%%%%%%%%%
%%%%%%%%%%%%%%%%%%%%%%%%%%%%%%%%%%%%%%%%%%%%%%%%%

\section{Introduction}\label{s:intro}
The goal of this paper is to classify toroidal solenoids defined by non-singular matrices with integer coefficients as introduced by M.~C.~McCord in 1965 \cite{m}. More precisely, let $\bT^n$ denote a torus considered as a quotient of  
$\bR^n$ by its subgroup $\bZ^n$. A matrix $A\in\operatorname{M}_n(\bZ)$ induces a map $A:\bT^n\rar \bT^n$,
$A\left([{\bf x}]\right)=[A{\bf x}]$, $[{\bf x}]\in\bT^n$, ${\bf x}\in\bR^n$. Consider the inverse system $(M_j,f_j)_{j\in\bN}$, where $f_j:M_{j+1}\rar M_{j}$,  
$M_j=\bT^n$ and $f_j=A$ for all $j\in\bN$.
The inverse limit $\mathcal{S}_A$ of the system is called a  ({\em toroidal}) {\em solenoid}. As a set, $\cS_A$ is a subset of 
$\prod_{j=1}^{\infty} M_j$ consisting of points $(z_j)\in\prod_{j=1}^{\infty} M_j$ such that $z_j\in M_j$ and $f_j(z_{j+1})=z_{j}$ for $\forall j\in\bN$, {\it i.e.},
$$
\cS_A=\left\{(z_j)\in\prod_{j=1}^{\infty}\bT^n\,\,\Big\vert\,\,z_j\in\bT^n,\,\, A(z_{j+1})=z_j,\,\,j\in\bN \right\}.
$$
Endowed with the natural group structure and the induced topology from the Tychonoff (product) topology on $\prod_{j=1}^{\infty} \bT^n$, $\cS_A$ is an  
$n$-dimensional topological abelian group. It is compact, metrizable, and connected, but not locally connected and not path connected. Toroidal solenoids are examples of inverse limit dynamical systems. When $n=1$ and $A=d$, $d\in\bZ$, solenoids  are called $d$-{\it adic solenoids} or {\it Vietoris solenoids}. The first examples were studied by L. Vietoris in 1927 for $d=2$
\cite{v} and later in 1930 by van Dantzig for an arbitrary $d$ \cite{d}. The problem of classifying  
toroidal solenoids (up to homeomorphisms) has been studied extensively based on their topological invariants and holonomy pseudogroup actions (see {\it e.g.}, \cite{c}  
and \cite{bl}). In \cite{s} and the present work, we employ a number-theoretic approach to solving the problem.

\sbr

It is known that the first \^Cech cohomology group $H^1(\cS_A,\bZ)$ of $\cS_A$ is isomorphic to a subgroup $G_{A^t}$ of $\bQ^n$ defined by the transpose $A^t$ of $A$ as follows:
%\bbe\label{eq:gaa}
$$
G_{A^t}=\left.\left\{(A^t)^{-k}{\bf x}\,\,\right\vert\, {\bf x}\in\bZ^n,\,k\in\bZ \right\}.
$$
%\ee
On the other hand, since $\cS_A$ is a compact connected abelian group, $H^1(\cS_A,\bZ)$ is isomorphic to the character group $\widehat{\cS_A}$ of $\cS_A$. Thus, for a non-singular $B\in\M_n(\bZ)$, using Pontryagin duality theorem, we see that $\cS_A$, $\cS_B$ are isomorphic as topological groups if and only if 
$G_{A^t}$, $G_{B^t}$ are isomorphic as abstract groups. Therefore, we study isomorphism classes of groups of the form $G_A$, where 
$A\in\M_n(\bZ)$ is non-singular. 

\sbr

If $n=1$, we have $A,B\in\bZ$ and 
$G_{A}$, $G_{B}$ are isomorphic if and only if $A$, $B$ have the same prime divisors. Note that if $A$, $B$ are conjugate by a matrix in $\GL_n(\bZ)$, then clearly $G_A$, $G_B$ are isomorphic (notationally, $G_A\cong G_B$). However, the converse is not true. In general, the class of matrices $A,B\in\M_n(\bZ)$ with isomorphic groups $G_A$, $G_B$ is much larger than the class of $\GL_n(\bZ)$-conjugate matrices. We have an example, where given an irreducible polynomial $h\in\bZ[x]$, there are three $\GL_2(\bZ)$-conjugacy classes of matrices with integer coefficients and characteristic polynomial $h$, 
but all three classes constitute just one class of isomorphic groups of the form $G_A$ \cite[Example 4]{s}.  It might also happen that $G_A\cong G_B$, but $A$, $B$ do not even share the same characteristic polynomial, so that $A$, $B$ are not conjugate by a matrix in 
$\GL_n(\bQ)$ (see {\it e.g.}, \cite[Example 2]{s}). In \cite{s} we classified groups $G_A$ in the case $n=2$. In the generic case, {\em i.e.}, when the characteristic polynomial of $A$ is irreducible, we linked $G_A$ to 
a generalized ideal class generated by an eigenvector of $A$ in the splitting field of the characteristic polynomial of $A$. We showed that if $G_A\cong G_B$, then the characteristic polynomials of $A$, $B$ share the same splitting field and, essentially, $G_A$ and $G_B$ are isomorphic if and only if the corresponding ideal classes are multiples of each other. 
It turns out that this is no longer true when $n>2$. In this paper, we finish the classification of groups $G_A$ (and hence, the associated toroidal solenoids 
$\cS_A$) for an arbitrary $n$. We provide necessary and sufficient conditions for 
$G_A\cong G_B$ for any $A,B\in\M_n(\bZ)$ and consider special cases as well. In particular, we formulate sufficient conditions under which $G_A\cong G_B$ if and only if the corresponding ideal classes are multiples of each other. We give examples illustrating how our theorems can be used to check whether $G_A\cong G_B$ for given $A,B\in\M_n(\bZ)$ in practice. 
We also consider applications of the obtained results to the  class of $\bZ^n$-odometers defined by matrices $A\in\M_n(\bZ)$. 

\noindent

\textbf{Acknowledgements.} The author thanks Mario Bonk for suggesting the problem and useful discussions. Support for this project was provided by PSC-CUNY Awards TRADA-51-133,  TRADB-53-92 jointly funded by The Professional Staff Congress and The City University of New York.
%The author also thanks the anonymous referee for suggesting valuable improvements to the paper.

\section{Localization}
For a non-singular $n\times n$-matrix $A$ with integer coefficients, $A\in\operatorname{M}_n(\bZ)$, define
\bbe\label{eq:opr}
G_A=\left.\left\{A^{-k}{\bf x}\,\right\vert\, {\bf x}\in\bZ^n,\,k\in\bZ \right\},\quad
\bZ^n\subseteq G_A\subseteq \bQ^n.
\ee
One can readily check that $G_A$ is a subgroup of $\bQ^n$.

\sbr

For a prime $p\in\bN$ denote
$$
\bZ_{(p)}=\left.\left\{\frac{m}{n}\in\bQ \,\, \right\vert \, m,n\in\bZ,\,\,n\ne 0,\,\,(p,n)=1\right\},
$$
a subring of $\bQ$. (Here $(p,n)$ denotes the greatest common divisor of $p$ and $n$.) Let $\bQ_p$ denote the field of $p$-adic numbers with the subgring of $p$-adic integers $\bZ_p$. 
For $N=\det A$, $N\in \bZ$, $N\ne 0$, let
\bbe\label{eq:rrr}
\caR=\bZ\left[\frac{1}{N} \right]=\left\{\frac{m}{N^k}\,\,\Big\vert\,\,m,k\in\bZ  \right\}
\ee
be the ring of $N$-adic rationals.

\begin{rem}\label{r:1}
Note that $G_A$ is a (additive)  subgroup of 
$\caR^n$,
since $A^{-k}=\frac{1}{(\det A)^k}\tilde A$, $k\in\bN$, with $\tilde A\in\M_n(\bZ)$. 
However, $G_A\ne\caR^n$ in general.
\end{rem}

\begin{lem}\label{lem:f1} For a prime $p\in\bN$ denote
$G_{A,p}=G_A\otimes_{\bZ}\bZ_{(p)}$. Then
$$
G_A=\bigcap_{p}G_{A,p}=\caR^n\bigcap_{p\,\vert\det A}G_{A,p}.
$$
Here $G_{A,p}$ is considered as a subset of $\bQ^n$.
\end{lem}
\begin{proof}
See \cite[p. 183, Lemma 93.1]{f2} for the first equality, which holds for any abelian subgroup of $\bQ^n$ and, more generally, for an abelian torsion free group of at most countable rank. Hence, taking into account Remark \ref{r:1}, we have $G_A\subseteq\caR^n\bigcap_{p\,\vert\det A}G_{A,p}$. The opposite inclusion 
is proved as in {\em loc.cit}. Namely, let $x\in\caR^n\bigcap_{p\,\vert\det A}G_{A,p}$. Then
$$
x=\sum x_ia_i,\quad x_i\in \bZ_{(p)},\,a_i\in G_A,
$$
and there exists $s\in\bZ$ coprime with $p$ such that $sx\in G_A.$ Since $x\in\caR^n$, there exists a power of $N$ such that $N^kx\in\bZ^n$, $k\in\bN$, $N=\det A$. Let $p_1,p_2,\ldots,p_l\in\bN$ be all the prime divisors of $N$. Since $x\in \bigcap_{p\,\vert\det A}G_{A,p}$, by above, for each $p_i$ there exists $s_i\in\bZ$ coprime with $p_i$ such that $s_ix\in G_A$. Since 
$N^k,s_1,s_2,\ldots,s_l$ are coprime and $\bZ^n\subset G_A$, we have $x\in G_A$. 
\end{proof}

For a prime $p\in\bN$ denote
$\overline{G}_{A,p}=G_A\otimes_{\bZ}\bZ_p$. Naturally, $\bZ_p^n\subseteq \overline{G}_{A,p}\subseteq \bQ_p^n$.

\begin{lem}\label{lem:f2} Let
$\overline{G}_{A,p}=G_A\otimes_{\bZ}\bZ_p$, 
$G_{A,p}=G_A\otimes_{\bZ}\bZ_{(p)}$. Then
$$
\bQ^n\cap \overline{G}_{A,p}=G_{A,p},
$$
where $\bQ^n\hrar\bQ_p^n$, and the intersection is in $\bQ_p^n$.
\end{lem}
\begin{proof}
See \cite[p. 183, Lemma 93.2]{f2}, \cite{d}. It is proved there that if $G$ is an abelian torsion free group of at most countable rank, then 
$$
(G\otimes_{\bZ}\bQ)\cap (G\otimes_{\bZ}\bZ_p)=G\otimes_{\bZ}\bZ_{(p)}.
$$
Apply the result to $G=G_A$ and note that $G_A\otimes_{\bZ}\bQ=\bQ^n$.
\end{proof}

\begin{cor}\label{cor:loc} 
$$
G_A=\bigcap_{p}\,(\bQ^n\cap\overline{G}_{A,p})=\bigcap_{p\,\vert\det A}(\caR^n\cap\overline{G}_{A,p}).
$$
\end{cor}

%\begin{cor}\label{cor:loc}
%$$
%G_A=\bQ^n\,\bigcap_{p}\,\overline{G}_{A,p}=\caR^n\bigcap_{p\,\vert\det A}\overline{G}_{A,p}.
%$$
%\end{cor}
\begin{proof}
Follows from Lemma \ref{lem:f1} and Lemma \ref{lem:f2}.
\end{proof}

\begin{prop}\cite[Prop. 3.8]{s}\label{prop:1} 
Let $A\in\M_n(\bZ)$ be non-singular, let $h_A\in\bZ[x]$ be the characteristic polynomial of $A$, and let $p\in\bZ$ be prime. 
Let $t_p=t_p(A)$ denote the multiplicity of zero in the reduction of $h_A$ modulo $p$, $0\leq t_p\leq n$. Then, as $\bZ_p$-modules,
$$\overline{G}_{A,p}\cong \bQ^{t_p}_p\oplus\bZ^{n-t_p}_p.$$
In particular,
\begin{enumerate}[$(1)$]
\item $p$ does not divide $\det A$ if and only if $\overline{G}_{A,p}= \bZ_p^n;$

\item $h_A\equiv x^n\,(\text{mod }p)$ if and only if $\overline{G}_{A,p}= \bQ_p^n$. 
 
\end{enumerate}
\end{prop}

Thus,
\bbe\label{eq:decom}
\overline{G}_{A,p}=D_p(A)\oplus R_p(A),
\ee
where $D_p(A)\cong \bQ_p^{t_p}$ denotes a divisible part of $\overline{G}_{A,p}$
and $R_p(A)\cong \bZ_p^{n-t_p}$ denotes a reduced $\bZ_p$-submodule of
$\overline{G}_{A,p}$.
Let
$$
\det A=ap_1^{s_1}p_2^{s_2}\cdots p_l^{s_l}
$$
be the prime-power factorization of $\det A$, where $p_1,p_2,\ldots,p_l\in\bN$ are distinct primes, $a=\pm 1$, and $s_1,s_2,\ldots,s_l\in\bN$. Let 
$$
\cP=\cP(A)=\{p_1,p_2,\ldots,p_l\}.
$$
The case $\cP=\emptyset$, equivalently, $A\in\GL_n(\bZ)$, has been settled as follows:
\begin{lem}\cite[Lemma 3.2]{s}\label{lem:det} Let $A,B\in\M_n(\bZ)$ be non-singular.
\begin{itemize}

\item[$(i)$] Assume $A\in\GL_n(\bZ)$. 
Then $G_A\cong G_B$ if and only if $B\in\GL_n(\bZ)$  if and only if $G_A= G_B=\bZ^n$.

\item[$(ii)$] Let $G_A\cong G_B$ and $A\not\in\GL_n(\bZ)$, {\it i.e.}, $\det A\ne\pm 1$. Then $\det B\ne\pm 1$ and $\det A,$ $\det B$ have the same prime divisors $($in $\bZ$$)$. 
\end{itemize}
\end{lem}
Therefore, for the rest of the paper we assume $\cP\ne\emptyset$.
 Denote
\begin{eqnarray*}
\cP'=\cP'(A)&=&\left\{p\in\cP,\,\,h_A\not\equiv x^n\,(\text{mod }p)\right\},
%\cS'(A)&=&\left\{p\in\cS\,\,\vert\,\, a_1\not\equiv 0\,(\text{mod }p),\ldots,a_{n-1}\not\equiv 0\,(\text{mod }p)\right\}=\\
%&=&\left\{p\in\bZ\text{ is a prime}\,\,\vert\,\,a_0\equiv 0\,(\text{mod }p),\,\,a_1\not\equiv 0\,(\text{mod }p),\ldots, a_{n-1}\not\equiv 0\,(\text{mod }p) \right\}.
\end{eqnarray*}
where $h_A\in\bZ[x]$ denotes the characteristic polynomial of $A$. The case 
$\cP'=\emptyset$ has been settled as well.
\begin{lem}\cite[Lemma 3.10]{s}
\label{lem:tt}
Let $A,B\in\M_n(\bZ)$ be non-singular and let $h_A,h_B\in\bZ[x]$ be their respective characteristic polynomials. 
Assume that for any prime $p\in\bN$ that divides $\det A$ we have 
$$
h_A\equiv x^n\,(\text{mod }p).
$$
Then $G_A\cong G_B$ $($with $T=I_n$$)$ if and only if $\det A$, $\det B$ have the same prime divisors and for any prime $p\in\bZ$ that divides $\det B$ we have 
$
h_B\equiv x^n\,(\text{mod }p).
$
\end{lem}

Therefore, for the rest of the paper we assume $\cP'\ne\emptyset$. 

\begin{rem}\label{r:p}
By Proposition \ref{prop:1},
for non-singular $A,B\in\M_n(\bZ)$, if $G_A\cong G_B$, then 
$\cP(A)=\cP(B)$, 
$\cP'(A)=\cP'(B)$, and $t_p(A)=t_p(B)$ for any prime $p\in\bN$. The converse is not true (see {\em e.g.}, \cite[Example 1]{s}, where non-singular $A,B\in\M_2(\bZ)$ share the same characteristic polynomial, but $G_A$ is not isomorphic to $G_B$).
\end{rem} 

\begin{cor}\label{cor:loc'}
$$
G_A=\bigcap_{p\,\in\,\cP'}(\caR^n\cap\overline{G}_{A,p}).
$$
\end{cor}
\begin{proof}
Follows from Corollary \ref{cor:loc}, since 
$$
\overline{G}_{A,p}=\bQ_p^n\,\,\text{ for any }p\in\cP\backslash \cP'
$$
by Proposition \ref{prop:1}.
\end{proof}

The following lemma provides an explicit basis for the decomposition of
$\overline{G}_{A,p}$ as in \eqref{eq:decom}. Let 
$t_p=t_p(A)$ denote the multiplicity of zero in the reduction of the characteristic polynomial of $A$ modulo $p$, $0\leq t_p\leq n$. Let
$$
\bZ(p^{\infty})=\bQ_p/\bZ_p
$$
denote the Pr\"ufer $p$-group.

\begin{lem}\label{l:repr}
Let $A\in\M_n(\bZ)$ be non-singular. For any $p\in \cP$ 
there exists $W_p\in\GL_n(\bZ_p)$ such that
\bbe\label{eq:1up}
W_p^{-1}AW_p=\left(\begin{matrix}
A_1 & * \\
0 & A_2
\end{matrix}
\right),
\ee
where $A_1\in\M_{t_p}(\bZ_p)$, $A_2\in\GL_{n-t_p}(\bZ_p)$, and $A_1$ has characteristic polynomial $h_1\in\bZ_p[x]$ with 
\bbe\label{eq:2up}
h_1\equiv x^{t_p}\,(\text{mod }p).
\ee
Let $W_p=\left(\begin{matrix}
{\bf w}_{p1}  & \ldots & {\bf w}_{pn}
\end{matrix}
\right)$, where ${\bf w}_{p1},\ldots, {\bf w}_{pn}\in\bZ_p^n$. Then
\begin{eqnarray}
D_p(A)&=&\Span_{\bQ_p}({\bf w}_{p1},\ldots,{\bf w}_{pt_p})\cong\bQ_p^{t_p},\label{eq:rand1}\\
R_p(A)&=&\Span_{\bZ_p}({\bf w}_{pt_p+1},\ldots,{\bf w}_{pn})\cong\bZ_p^{n-t_p}.\label{eq:rand2}
\end{eqnarray}
In particular,
\bbe\label{eq:minloc}
\overline{G}_{A,p}/\bZ_p^n\cong \bZ(p^{\infty})^{t_p}.
\ee
\end{lem}
\begin{proof}
One can show that for an irreducible polynomial $\chi\in\bZ_p[x]$ of degree 
$n$, either $p$ does not divide 
$\chi(0)$ or $\chi\equiv x^n\,(\text{mod }p)$ (see, {\em e.g.}, the proof of \cite[Prop. 3.8]{s}).  Therefore,
the existence of $W_p\in\GL_n(\bZ_p)$ satisfying \eqref{eq:1up} and 
\eqref{eq:2up} follows from Theorem \ref{th:trig} below. Moreover, the proof of 
Theorem \ref{th:trig} gives an algorithm to construct $W_p$.
Let $\tilde A=W_p^{-1}AW_p$, $\tilde A\in\M_n(\bZ_p)$, and
$$
%\bbe\label{eq:locdecoma}
G_{\tilde A}=\left.\left\{\tilde A^{-k}{\bf x}\,\right\vert\, {\bf x}\in \bZ_p^n,\,k\in\bZ \right\}= \bQ_p{\bf e}_1\oplus\cdots\oplus \bQ_p{\bf e}_{t_p}\oplus \bZ_p{\bf e}_{t_p+1}\oplus\cdots\oplus \bZ_p{\bf e}_{n},
%\ee
$$
{\it i.e.}, with respect to the standard basis $\{{\bf e}_1,\ldots,{\bf e}_n \}$ of 
$\bQ_p^n$,
$$
G_{\tilde A}=\bQ_p^{t_p}\oplus \bZ_p^{n-t_p}
$$
(this follows, {\em e.g.}, from Proposition \ref{prop:1} applied to $A_1$ and $A_2$).
Since $W_p\in\GL_n(\bZ_p)$, we have $\overline{G}_{A,p}=W_p(G_{\tilde A})$ and
 $\{W_p{\bf e}_1,\ldots,W_p{\bf e}_n \}$ is a free $\bZ_p$-basis of $\bZ_p^n$, ${\bf w}_{pi}=W_p{\bf e}_i$, $1\leq i\leq n$. Hence, \eqref{eq:rand1} -- \eqref{eq:minloc} follow.
\end{proof}

\section{Minimax groups}

\begin{defin}\cite{gm}
A torsion-free abelian group $G$ of rank $n$ is called a {\em minimax} group if there exists a free subgroup $H$ of $G$ of rank $n$ such that
$$
G/H\cong \bigoplus_{i=1}^l\bZ(p_i^{\infty})^{t_i},
$$
where $p_1,p_2,\ldots,p_l\in\bN$ are distinct primes and $t_1,t_2,\ldots,t_l\in\bN$.
\end{defin}

Let $A\in\operatorname{M}_n(\bZ)$ be non-singular.
We show that $G_A$ defined by \eqref{eq:opr} is a minimax group in the lemma below. Let $h_A\in\bZ[x]$ denote the characteristic polynomial of $A$.

\begin{lem}\label{l:minmax}
$G_A$ is a minimax group. Namely,
$$
G_A/\bZ^n\cong \bigoplus_{i=1}^l\bZ(p_i^{\infty})^{t_i},
$$
where $p_1,p_2,\ldots,p_l\in\bN$ are all distinct prime divisors of $\det A$, and $t_i$ is the multiplicity of zero in the reduction of $h_A$ modulo $p_i$, $0< t_i\leq n$, $1\leq i\leq l$.
\end{lem}

\begin{proof}
Let $p\in\bN$ be prime, and let $x=x_0+x_1\in\bQ_p$, where $x_1\in\bZ_p$ and $x_0\in\bQ$ is a ``fractional" part of $x$. It is well-known that the correspondence 
$\phi_p(x)=x_0$ induces a well-defined injective homomorphism $\phi_p:\bQ_p/\bZ_p\hrar\bQ/\bZ$ and that $\phi=\bigoplus_p\phi_p$ is a group isomorphism
$$
\phi=\bigoplus_p\phi_p:\bigoplus_p\bQ_p/\bZ_p\rarab{\sim}\bQ/\bZ.
$$
Let
%By composing $\phi^n$ with certain natural isomorphisms as follows
%$$
%\bigoplus_p\bQ_p^n/\bZ_p^n\rarab{\sim}\bigoplus_p\left(\bQ_p/\bZ_p\right)^n\rarab{\sim}\left(\bigoplus_p\bQ_p/\bZ_p\right)^n
%\rarab{\phi^n}\left(\bQ/\bZ\right)^n\rarab{\sim}\bQ^n/\bZ^n,
%$$
%we get an isomorphism 
$$
\psi:\bigoplus_p\bQ_p^n/\bZ_p^n\rarab{\sim}\bQ^n/\bZ^n
$$
be the natural isomorphism induced by $\phi$. It restricts to an isomorphism
$$
\psi_A:\bigoplus_p \overline{G}_{A,p}/\bZ_{p}^n\rarab{\sim} G_A/\bZ^n.
$$
Indeed, recall that $A$ has integer entries and therefore, multiplication by $A^i$ commutes with $\psi$ for any non-negative integer $i$. 
Furthermore, ${\bf u}\in\overline{G}_{A,p}$ (resp., ${\bf v}\in G_A$) if and only if $A^k{\bf u}\in\bZ_p^n$ (resp., $A^k{\bf v}\in\bZ^n$) for some $k\in\bN\cup\{0\}$. 
Finally, $\overline{G}_{A,p}/\bZ_{p}^n$ is trivial for any $p$ that does not divide $\det A$ by Proposition \ref{prop:1}. Therefore, $\psi_A$ is an isomorphism between the following groups
\bbe\label{eq:isom}
\psi_A:\bigoplus_{i=1}^l \overline{G}_{A,p_i}/\bZ_{p_i}^n\rarab{\sim} G_A/\bZ^n.
\ee
Combined with \eqref{eq:minloc}, this proves the lemma.
\end{proof}

Using 
Lemma \ref{l:repr} and isomorphism $\psi_A$ in  \eqref{eq:isom}, one can now write down (infinitely many) group generators of $G_A$ ({\em c.f.}, \cite{gm}). 
%??? do we need the following lemma ???

\begin{lem}\label{l:repr'}
Let $A\in\M_n(\bZ)$ be non-singular. 
For each $p\in \cP'$, let 
$$
W_p=\left(\begin{matrix}
{\bf w}_{p1}  & \ldots & {\bf w}_{pn}
\end{matrix}
\right)
$$ 
be as in Lemma $\ref{l:repr}$, ${\bf w}_{pj}\in\bZ_p^n$, $1\leq j\leq n$. Then
\bbe\label{eq:gen-n}
G_A=<{\bf e}_1,\ldots,{\bf e}_n, q^{-\infty}{\bf e}_1, \ldots, q^{-\infty}{\bf e}_n,
p^{-\infty}{\bf w}_{p1},\ldots,p^{-\infty}{\bf w}_{pt_p}>,
\ee
i.e., 
$G_A$ is generated over $\bZ$ by ${\bf e}_1,\ldots,{\bf e}_n$, 
$q^{-s}{\bf e}_1,\ldots,q^{-s}{\bf e}_n$, and $p^{-k}{\bf w}^{(k)}_{pi}$, where ${\bf w}^{(k)}_{pi}$ is the $(k-1)$-st partial sum of the standard $p$-adic expansion of ${\bf w}_{pi}$, $k,s\in\bN$, $1\leq i\leq t_p$, $q\in\cP\backslash\cP'$, $p\in\cP'$.
\end{lem}
\begin{proof}
Let $h_A\in\bZ[x]$ denote the characteristic polynomial of $A$.
Let $q\in\cP\backslash\cP'$, {\it i.e.}, $q\in\bN$ is a prime such that 
$$h_A\equiv x^{n}\,(\text{mod }q),\quad t_q=n.$$
By Proposition \ref{prop:1}, we have $\overline{G}_{A,q}=\bQ_q^n$. 
Then $\overline{G}_{A,q}$ is generated over $\bZ_q$ by
${\bf e}_1,\ldots, {\bf e}_n$,
$q^{-s}{\bf e}_1,\ldots, q^{-s}{\bf e}_n$, where $s\in\bN$, {\it i.e.}, in our notation,
\bbe\label{eq:local1'}
\overline{G}_{A,q}=\Span_{\bZ_q}({\bf e}_1,\ldots, {\bf e}_n, q^{-\infty}{\bf e}_1,\ldots,q^{-\infty}{\bf e}_n).
\ee
For $p\in\cP'$, by \eqref{eq:rand1}, \eqref{eq:rand2},
$\overline{G}_{A,p}$ is generated over $\bZ_p$ by ${\bf e}_1,\ldots,{\bf e}_n$, $p^{-k}{\bf w}_{pi}$, where $k\in\bN$, {\it i.e.},
\bbe\label{eq:local2'}
\overline{G}_{A,p}=\Span_{\bZ_p}({\bf e}_1,\ldots,{\bf e}_n, p^{-\infty}{\bf w}_{p1},\ldots,p^{-\infty}{\bf w}_{pt_p}).
\ee
Applying isomorphism $\psi_A$ in \eqref{eq:isom} to the generators of $\overline{G}_{A,p}$ in \eqref{eq:local1'} and \eqref{eq:local2'}, we get \eqref{eq:gen-n}.
\end{proof}

Generators of $G_A$ in \eqref{eq:gen-n} are written in terms of the standard basis $\{{\bf e}_1,\ldots, {\bf e}_n\}$ and 
vectors $\{{\bf w}_{p1},\ldots, {\bf w}_{pn}\}$, $p\in\cP'$. In what follows, we show the existence of a free basis $\{{\bf f}_1,\ldots, {\bf f}_n\}$ of $\bZ^n$ (that does not depend on $p$) and $p$-adic integers $\al_{pij}\in\bZ_p$ with 
$1\leq i\leq t_p$, $t_p+1\leq j\leq n$, $p\in\cP'$, that determine generators of $G_A$. It is often useful to extend constants from $\bQ$ to a number field $K$, a  finite extension of $\bQ$, {\it i.e.}, to consider $G_A\otimes_\bZ{\cO_K}$,
where $\cO_K$ denotes the ring of integers of $K$ (see Remark \ref{r:char} below). Therefore, we start with a preliminary result, which holds over $K$.

\begin{lem}\label{l:vect}
Let $\cS$ be a finite set of primes in $\bN$, let $K$ be a number field, and $n\in\bN$. For each $p\in\cS$ let $\fp$ be a prime ideal of $\cO_K$ above $p$ and let $V_{\fp}$ denote a non-zero proper subspace of $K_{\fp}^n$, where $K_{\fp}$ is the completion of $K$ with respect to $\fp$, $\dim V_{\fp}=t_p$, $0<t_p< n$. There exists a basis 
$\{{\bf f}_1,\ldots, {\bf f}_n\}$  of $\bZ^n$ such that for any $p\in\cS$ there are $\al_{pij}\in\cO_{\fp}$, 
$1\leq i\leq t_p<j\leq n$, such that
\bbe\label{eq:form}
{\bf x}_{pi}={\bf f}_i+\sum_{j=t_p+1}^n\al_{pij}{\bf f}_j,\quad 1\leq i\leq t_p,
\ee
is a $K_{\fp}$-basis of $V_{\fp}$. $($Here, $\cO_{\fp}$ denotes the ring of integers of $K_{\fp}.)$  
\end{lem}

\begin{proof}
It is a straightforward generalization of \cite[p. 194, Lemma 1]{gm} from $\bQ$ to a number field. 
We repeat their argument in order to use later in specific examples.
The argument does not depend on the choice of prime ideals $\fp$. Therefore, for simplicity, we denote 
$\cO_p=\cO_{\fp}$, $K_p=K_{\fp}$, $V_p=V_{\fp}$, and so on. 

\sbr

For a fixed $p\in\cS$ let ${\bf y}_{p1},\ldots,{\bf y}_{pt_p}$ be a $K_{p}$-basis of $V_{p}$. Let $(\pi)$ be the maximal ideal of $\cO_p$. Let 
$$
{\bf y}_{pi}=\sum_{k=1}^n\ga_{pi}^k{\bf e}_k,
$$
where $\forall\ga_{pi}^k\in K_p$. By multiplying or dividing by positive powers of $\pi$ if necessary, without loss of generality, we can assume that $\forall\ga_{pi}^k\in \cO_p$ and for $\forall i$ there is a unit among 
$\ga_{pi}^1,\ldots, \ga_{pi}^n$. Let $\{{\bf f}_1,\ldots,{\bf f}_n\}$ be an ordered basis of $\bZ^n$ obtained by permuting elements in the standard basis $\{{\bf e}_1,\ldots,{\bf e}_n\}$, so that 
\bbe\label{eq:p}
{\bf y}_{p1}=\sum_{k=1}^n\de_{p1}^k{\bf f}_k,\quad \de_{p1}^1\in\cO_p^{\times}.
\ee
Here $\cO_p^{\times}$ denotes the set of all units in $\cO_p$. Now we show that, without loss of generality, we can assume $\de_{q1}^1\in\cO_q^{\times}$ for any $q\in\cS$ other than $p$. Indeed,
denote by $\Ga$ the set of all primes $q\in\cS$ such that $\de_{q1}^1\in\cO_q^{\times}$. By \eqref{eq:p}, $\Ga\ne\emptyset$ and let
%and also, without loss of generality, we can assume that $\de_{pj1}=0$ for any $p\in\Ga$ and any $j\in\{1,\ldots,l_p\}$. 
%Let $t\in\bZ$ be the product of all primes in $\Ga$, {\it i.e.},
\bbe\label{eq:t}
t=\prod_{p\in\Ga} p.
\ee
Let $s\in\cS\backslash\Ga$, {\it i.e.}, $\de_{s1}^1\in \cO_s$ is not a unit. 
By assumption, there is $j\in\{2,\ldots,n\}$ such that $\de_{s1}^j\in\cO_s^{\times}$. Consider ${\bf f}'_j={\bf f}_j-t{\bf f}_1$ and ${\bf f}_i'={\bf f}_i$ for any $i\ne j$, $1\leq i\leq n$. Then, with respect to the new basis $\{{\bf f}'_1,\ldots,{\bf f}'_n\}$ of $\bZ^n$, we have
$$
{\bf y}_{p1}=\sum_{k=1}^n\tilde\de_{p1}^k{\bf f}'_k,\quad \tilde\de_{p1}^1\in\cO_p^{\times}
$$
for any $p\in\Ga$ and $p=s$. We now  add $s$ to $\Ga$ and change $t$ 
in \eqref{eq:t} to $ts$. Repeating the process for the remaining elements in 
$\cS\backslash\Ga$, we obtain a basis 
$\{{\bf f}''_1,\ldots,{\bf f}''_n\}$ of $\bZ^n$ such that for any $p\in\cS$ we have
$$
{\bf y}_{pi}=\sum_{k=1}^n\epsilon_{pi}^k{\bf f}''_k,\quad \epsilon_{p1}^1\in\cO_p^{\times}, \quad 
\epsilon_{pi}^2,\ldots,\epsilon_{pi}^{n}\in\cO_p,\quad 1\leq i\leq t_p.
$$
By dividing ${\bf y}_{p1}$ by $\epsilon_{p1}^1$, without loss of generality, 
$\epsilon_{p1}^1=1$. Let $\cS'=\{p\in\cS,\,\,t_p\geq 2\}$.
For any $p\in\cS'$ and $2\leq i\leq t_p$, let
$\tilde{\bf y}_{pi}={\bf y}_{pi}-\epsilon_{pi}^1{\bf y}_{p1}$. Then  
$$
\tilde{\bf y}_{pi}\in\Span_{\cO_p}({\bf f}''_2,\ldots,{\bf f}''_n), \quad
2\leq i\leq t_p,\,\, p\in\cS'.  
$$
Applying induction to vectors $\tilde{\bf y}_{pi}$, 
$2\leq i\leq t_p$, $p\in\cS'$, we get a free $\bZ$-basis $\{{\bf g}_2,\ldots,{\bf g}_n\}$ of 
$\Span_{\bZ}({\bf f}''_2,\ldots,{\bf f}''_n)$ such that
\begin{eqnarray*}
&& \hat{\bf y}_{pi}={\bf g}_i+\sum_{j=t_p+1}^n\mu_{pi}^j{\bf g}_j,\quad
\mu_{pi}^{t_p+1},\ldots,\mu_{pi}^n\in\cO_p,\,\,
2\leq i\leq t_p, \\
&& \Span_{\cO_p}(\tilde{\bf y}_{p2},\ldots,\tilde{\bf y}_{pt_p})
=\Span_{\cO_p}(\hat{\bf y}_{p2},\ldots,\hat{\bf y}_{pt_p}),\,\,p\in\cS'.
\end{eqnarray*}
%$$
%\Span_{\bZ}(p^{-\infty}\tilde{\bf w}_{pi},\quad
%2\leq i\leq t_p,\,\, p\in\cP'')=\Span_{\bZ}(p^{-\infty}\hat{\bf w}_{pi},\quad
%2\leq i\leq t_p,\,\, p\in\cP'').
%$$
Finally, for any $p\in\cS$ let 
\begin{eqnarray*}
&& \tilde{\bf y}_{p1}={\bf f}''_1+\sum_{k=2}^n\mu_{p1}^k{\bf g}_k,\quad 
\mu_{p1}^2,\ldots,\mu_{p1}^n\in\cO_p,\\
&& \hat{\bf y}_{p1}=\tilde{\bf y}_{p1}-\sum_{k=2}^{t_p}\mu_{p1}^k\hat{\bf y}_{pk}=
{\bf f}''_1+\sum_{j=t_p+1}^n\tilde\mu_{p1}^j{\bf g}_j.
\end{eqnarray*}
Hence, with respect to the $\bZ$-basis $\{{\bf f}''_1,{\bf g}_2,\ldots,{\bf g}_n \}$, ${\bf x}_{pi}=\hat{\bf y}_{pi}$, $1\leq i\leq t_p$, $p\in\cS$, have the form  \eqref{eq:form}. 
\end{proof}

In the next lemma we apply Lemma \ref{l:vect} to divisible parts of $\overline{G}_{A,p}$ and more generally, to the divisible parts of $\overline{G}_{A,p}\otimes_{\bZ_p}\cO_{\fp}$. The result is a free basis $\{{\bf f}_1,\ldots, {\bf f}_n\}$ of $\bZ^n$ and numbers $\al_{pij}\in\bZ_p$, $p\in\cP'(A)$, that produce generators of $G_A$ over $\bZ$.

\begin{lem}\label{lem:gen-n}
Let $A\in\M_n(\bZ)$ be non-singular and let $K$ be a number field. 
There exists a basis 
$\{{\bf f}_1,\ldots, {\bf f}_n\}$ of $\bZ^n$ such that for any $p\in\cP'(A)$ and a prime ideal $\fp$ of $\cO_K$ above $p$ there are $\al_{pij}\in\cO_{\fp}$, 
$i\in\{1,\ldots,t_p\}$, $j\in\{t_p+1,\ldots,n\}$, such that
\begin{eqnarray}
&& \overline{G}_{A,p}\otimes_{\bZ_p}\cO_{\fp}=\Span_{K_{\fp}}({\bf x}_{p1},\ldots,{\bf x}_{pt_p})\oplus\Span_{\cO_{\fp}}({\bf f}_{t_p+1},\ldots, 
{\bf f}_{n}),\label{eq:ppp} \\
&& {\bf x}_{pi}={\bf f}_i+\sum_{j=t_p+1}^n\al_{pij}{\bf f}_j, \quad 1\leq i\leq t_p. \label{eq:ppp1}
\end{eqnarray}
Moreover, all $\al_{pij}$ belong to $\bZ_p$, they do not depend on $K$, $\fp$ above $p$, and are uniquely defined for a fixed ordered basis 
$\{{\bf f}_1,\ldots, {\bf f}_n\}$. Furthermore,
\bbe\label{eq:gener-n0}
G_A=<{\bf f}_1,\ldots,{\bf f}_n, q^{-\infty}{\bf f}_1, \ldots, q^{-\infty}{\bf f}_n, 
p^{-\infty}{\bf x}_{pi}>,\quad q\in\cP\backslash\cP',\quad 1\leq i\leq t_p.
\ee
\end{lem}

\begin{proof} By Lemma \ref{eq:decom}, $\overline{G}_{A,p}=D_p(A)\oplus R_p(A)$, where as $\bZ_p$-modules, $D_p(A)\cong \bQ_p^{t_p}$, 
$R_p(A)\cong \bZ_p^{n-t_p}$. Denote 
\begin{eqnarray*}
&& \overline{G}_{\fp}=\overline{G}_{A,p}\otimes_{\bZ_p}\cO_{\fp}, \\
&& D_{\fp}=D_{p}(A)\otimes_{\bZ_p}\cO_{\fp}, \\
&& R_{\fp}=R_{p}(A)\otimes_{\bZ_p}\cO_{\fp}. 
\end{eqnarray*}
Then
$\overline{G}_{\fp}=D_{\fp}\oplus R_{\fp}$, where as $\cO_{\fp}$-modules, $D_{\fp}\cong K_{\fp}^{t_p}$, $R_{\fp}\cong \cO_{\fp}^{n-t_p}$. We apply Lemma \ref{l:vect} to $\cS=\cP'(A)$, $K$, and $V_{\fp}=D_{\fp}$. Then there exists a basis 
$\{{\bf f}_1,\ldots, {\bf f}_n\}$ of $\bZ^n$ such that for any $p\in\cP'(A)$, 
$D_{\fp}=\Span_{K_{\fp}}({\bf x}_{p1},\ldots,{\bf x}_{pt_p})$, and ${\bf x}_{pi}$'s are given by \eqref{eq:form}. 
 We only need to show 
$\overline{G}_{\fp}\subseteq D_{\fp}\oplus\Span_{\cO_{\fp}}({\bf f}_{t_p+1},\ldots, 
{\bf f}_{n})$. Indeed, by \eqref{eq:rand2}, for any ${\bf u}\in R_{\fp}$, 
$$
{\bf u}=\sum_{i=t_p+1}^n\al_i{\bf w}_{pi}=\sum_{i=1}^n\be_i{\bf f}_i=
\sum_{i=1}^{t_p}\ga_i{\bf x}_{pi}+\sum_{i=t_p+1}^n\ga_i{\bf f}_i,
$$
where all $\al_i\in\cO_{\fp}$ by definition of $R_{\fp}$, and all $\be_i\in\cO_{\fp}$, since all ${\bf w}_{pi}\in\bZ_p^n$.
Finally, all $\ga_i\in\cO_{\fp}$ by definition of ${\bf x}_{pi}$. This proves \eqref{eq:ppp}.

%\sbr
%
%Let $\cP'(A)=\{p_1,\ldots,p_l\}$. 
%Assume that for a fixed number field $K$ and a fixed set $\cT$ of primes $\fp_1,\ldots,\fp_l$ in $\cO_K$ above $p_1,\ldots,p_l$, respectively, there are two sets $\al_{pij},\al'_{pij}\in\cO_{\fp}$ corresponding to the same ordered basis 
%$\{{\bf f}_1,\ldots, {\bf f}_n\}$ of $\bZ^n$ via \eqref{eq:ppp}, \eqref{eq:ppp1}, $p\in\cP'(A)$, $\fp\in\cT$. Let ${\bf x}_{pi}={\bf f}_i+\sum_{j=t_p+1}^n\al_{pij}{\bf f}_j$,
%${\bf x}'_{pi}={\bf f}_i+\sum_{j=t_p+1}^n\al'_{pij}{\bf f}_j$, $1\leq i\leq t_p$. Then ${\bf x}_{pi},{\bf x}'_{pi}\in D_{\fp}$ and hence 
%$$
%{\bf x}_{pi}-{\bf x}'_{pi}=\sum_{j=t_p+1}^n(\al_{pij}-\al'_{pij}){\bf f}_j\in D_{\fp}\cap \Span_{\cO_{\fp}}({\bf f}_{t_p+1},\ldots, {\bf f}_{n})=\{0\},
%$$
%hence $\al_{pij}=\al'_{pij}$ for all $i,j$, and $p$. 

\sbr

We now show that for any $K$, all $\al_{pij}\in\bZ_p$. 
By enlarging $K$ if necessary, without loss of generality, we assume $K$ is Galois over $\bQ$. Let $p\in\cP'(A)$ be arbitrary. By above, \eqref{eq:ppp}, \eqref{eq:ppp1} hold. For any $\sg\in\Gal(K_{\fp}/\bQ_p)$, we have 
$\sg(\overline{G}_{\fp})=\overline{G}_{\fp}$, $\sg(R_{\fp})=R_{\fp}$, and 
$\sg(D_{\fp})=\Span_{K_{\fp}}(\sg({\bf x}_{pi}))$, where
$$
\sg({\bf x}_{pi})={\bf f}_i+\sum_{j=t_p+1}^n\sg(\al_{pij}){\bf f}_j, \quad 1\leq i\leq t_p,
$$
since $A$, ${\bf f}_1,\ldots, {\bf f}_n$ are defined over $\bZ$. By the uniqueness of the divisible part, we have
$\Span_{K_{\fp}}(\sg({\bf x}_{pi}))=\Span_{K_{\fp}}({\bf x}_{pi})$ and hence $\sg(\al_{pij})=\al_{pij}$ for any $i,j$. Since $\al_{pij}\in\cO_{\fp}$, this implies $\al_{pij}\in\bZ_p$ and hence ${\bf x}_{pij}\in\bZ_p^n$ for all $p,i,j$. Furthermore, 
$\overline{G}_{A,p}$ consists of elements in $G_{\fp}$ invariant under the action of $\Gal(K_{\fp}/\bQ_p)$. Hence,
$$
\overline{G}_{A,p}=\Span_{\bQ_{\fp}}({\bf x}_{p1},\ldots,{\bf x}_{pt_p})\oplus\Span_{\bZ_p}({\bf f}_{t_p+1},\ldots,{\bf f}_{n}).
$$
On the other hand, if \eqref{eq:ppp}, \eqref{eq:ppp1} hold for $K=\bQ_p$ and the same basis ${\bf f}_1,\ldots, {\bf f}_n$, then
\begin{eqnarray*}
&&\overline{G}_{A,p}=\Span_{\bQ_{\fp}}({\bf x}'_{p1},\ldots,{\bf x}'_{pt_p})\oplus\Span_{\bZ_p}({\bf f}_{t_p+1},\ldots,{\bf f}_{n}), \\
&& {\bf x}'_{pi}={\bf f}_i+\sum_{j=t_p+1}^n\al'_{pij}{\bf f}_j, \quad 1\leq i\leq t_p,
\end{eqnarray*}
for some $\al'_{pij}\in\bZ_p$, a priori, different from $\al_{pij}\in\bZ_p$.
As above, by the uniqueness of the divisible part, we have $\al_{pij}=\al'_{pij}$ for all $p,i,j$. 
This shows that $\al_{pij}$'s do not depend on $K$ and $\fp$'s.

\sbr

For each $p\in\cP'(A)$ let ${\bf w}_{p1},\ldots,{\bf w}_{pt_p}$ be as in Lemma \ref{l:repr}. By \eqref{eq:rand1},
$\{{\bf w}_{p1},\ldots,{\bf w}_{pt_p}\}$ is a $\bQ_p$-basis of $D_p(A)$. By Lemma \ref{l:vect} applied to $\cS=\cP'(A)$,
$K=\bQ$, and $V_p=D_p(A)$, we get 
$\Span_{\bQ_p}({\bf w}_{p1},\ldots,{\bf w}_{pt_p})=\Span_{\bQ_p}({\bf x}_{p1},\ldots,{\bf x}_{pt_p})$. 
Thus, \eqref{eq:gener-n0} follows from \eqref{eq:gen-n}.
\end{proof}

\begin{defin}\cite{gm}
Let $\{{\bf f}_1,\ldots, {\bf f}_n\}$ and $\al_{pij}\in\bZ_p$ be as in  Lemma \ref{lem:gen-n}.
The set
$$
M(A;{\bf f}_1,\ldots, {\bf f}_n)=\{\al_{pij}\in\bZ_p\,\vert\,p\in\cP',\,1\leq i\leq t_p<j\leq n\}
$$
is called the {\it characteristic} of $G_A$ relative to the ordered basis $\{{\bf f}_1,\ldots, {\bf f}_n\}$. 
\end{defin}

\begin{rem}\label{rem:base}
To calculate a characteristic of $G_A$ in practice, one can start with a basis $\cW_p=\{{\bf w}_{p1},\ldots,{\bf w}_{pt_p}\}$ of the divisible part $D_p(A)$, and then apply the procedure in the proof of Lemma \ref{l:vect} for $\cS=\cP'(A)$, 
$K=\bQ$, $V_{\fp}=D_p(A)$ 
(see Lemma \ref{l:repr} for the definition of $\cW_p$). In turn, to find $\cW_p$, one can use the procedure described in the proof of Theorem \ref{th:trig} below.  
\end{rem}

Our ultimate goal is to characterize when $G_A\cong G_B$ for non-singular $A,B\in\M_n(\bZ)$. In the next lemma we show that by conjugating $A$ by a matrix in 
$\GL_n(\bZ)$ corresponding to $\{{\bf f}_1,\ldots, {\bf f}_n\}$, without loss of generality, we can assume that the
 characteristics of both $G_A$, $G_B$ are given with respect to the standard basis $\{{\bf e}_1,\ldots, {\bf e}_n\}$.

\begin{lem}\label{l:stand}
Let $A\in\M_n(\bZ)$ be non-singular and let
$M(A;{\bf f}_1,\ldots, {\bf f}_n)$
be the characteristic of $G_A$ relative to a free $\bZ$-basis $\{{\bf f}_1,\ldots, {\bf f}_n\}$ of $\bZ^n$. Let 
$\{{\bf g}_1,\ldots, {\bf g}_n\}$ be another free $\bZ$-basis of $\bZ^n$ and let
$S\in\GL_n(\bZ)$ be a change-of-basis matrix: $S{\bf f}_i={\bf g}_i$, 
$1\leq i\leq n$. Then $S(G_A)=G_{SAS^{-1}}$, $\cP'(A)=\cP'(SAS^{-1})$, $t_p(A)=t_p(SAS^{-1})$, and
$$
M(SAS^{-1};{\bf g}_1,\ldots, {\bf g}_n)=M(A;{\bf f}_1,\ldots, {\bf f}_n).
$$
\end{lem}
\begin{proof}
Follows easily from the definition \eqref{eq:opr} of $G_A$ and Lemma \ref{lem:gen-n}.
\end{proof}

\begin{lem}\label{lem:gen-n-bel}
Let $A\in\M_n(\bZ)$ be non-singular and let 
$$
M(A;{\bf f}_1,\ldots, {\bf f}_n)=\{\al_{pij}\,\vert\,p\in\cP',\,1\leq i\leq t_p<j\leq n\}
$$ 
be the characteristic of $G_A$ relative to a free basis $\{{\bf f}_1,\ldots, {\bf f}_n\}$ of $\bZ^n$. For ${\bf b}\in\bQ^n$ 
let
${\bf b}=\sum_{k=1}^nb_k{\bf f}_k$, $b_1,\ldots,b_n\in\bQ$. Then 
${\bf b}\in\overline{G}_{A,p}$ for $p\in\cP'$  if and only if 
\bbe\label{eq:b}
b_j-\sum_{i=1}^{t_p}b_i\al_{pij}\in\bZ_p,\,\,t_p+1\leq j\leq n.
\ee
Moreover, 
${\bf b}\in G_A$ if and only if $b_1,\ldots,b_n\in\caR$ and $\eqref{eq:b}$ holds 
for any $p\in\cP'$.
\end{lem}
\begin{proof}
It follows easily from \cite[p. 195, Lemma 2]{gm}. 
We repeat the argument adapted to our case. Since $\{{\bf f}_1,\ldots, {\bf f}_n\}$ is a free $\bZ$-basis of $\bZ^n$,
it follows from Corollary \ref{cor:loc'} that ${\bf b}\in G_A$ if and only if 
$b_1,\ldots,b_n\in\caR$ and ${\bf b}\in \overline{G}_{A,p}$ for any $p\in\cP'$.
Since 
$\bZ_p^n\subseteq \overline{G}_{A,p}$,
by Lemma \ref{lem:gen-n}, 
$\{ {\bf x}_{p1},\ldots,{\bf x}_{pt_p},
{\bf f}_{t_p+1},\ldots,{\bf f}_n \}$ is  a basis of $\bQ_p^n$ as a $\bQ_p$-vector space.
Thus,
\bbe\label{eq:loc}
{\bf b}=\sum_{k=1}^nb_k{\bf f}_k=\sum_{i=1}^{t_p}y_i{\bf x}_{pi}+
\sum_{j=t_p+1}^ny_j{\bf f}_j,\quad y_1,\ldots,y_n\in\bQ_p.
\ee
Hence, by Lemma \ref{lem:gen-n} applied to $K=\bQ$, ${\bf b}\in \overline{G}_{A,p}$ if and only if $y_{t_p+1},\ldots,y_n\in\bZ_p$. 
 Comparing coefficients in \eqref{eq:loc} and using \eqref{eq:form},
each $y_i=b_i$ and each $y_j=b_j-\sum_{i=1}^{t_p}b_i\al_{pij}$, hence \eqref{eq:b}.
\end{proof}
We are interested in studying isomorphism classes of groups $G_A$, {\it i.e.,}
when $G_A\cong G_B$ for non-singular $A,B\in\M_n(\bZ)$. If $n=1$, we have $A,B\in\bZ$ and $G_A\cong G_B$ if and only if $A$, $B$ have the same prime divisors in $\bZ$. Therefore, for the rest of the paper we assume $n\geq 2$.

\sbr 

The next result is a criterion for $G_A$, $G_B$ to be isomorphic. It is based on the facts that any isomorphism $\phi$ between $G_A$ and $G_B$ is induced by a matrix $T\in\GL_n(\bQ)$ (\cite[Lemma 3.1]{s}), 
$\phi$ induces a 
$\bZ_p$-module isomorphism between $\overline{G}_{A,p}$ and 
$\overline{G}_{B,p}$ for any prime $p\in\bN$, and, therefore, $\phi$ restricts to an isomorphism between the divisible parts $D_p(A)$, $D_p(B)$ (see \eqref{eq:rand1} for the definition). 

\sbr

Let $A,B\in\M_n(\bZ)$ be non-singular.
Define %$\caR$ is defined by \eqref{eq:rrr}, {\it i.e.},
$$
\caR(A)=\bZ\left[\frac{1}{N} \right]=\left\{\frac{x}{N^k}\,\,\Big\vert\,\,x,k\in\bZ  \right\},\quad N=\det A.
$$
By Lemma \ref{l:stand}, without loss of generality, we can assume that we have the characteristics of $G_A$, $G_B$ with respect to the same standard basis $\{{\bf e}_1,\ldots, {\bf e}_n\}$, {\it i.e.},
\begin{eqnarray}\label{eq:char}
&&M(A;{\bf e}_1,\ldots, {\bf e}_n)=\{\al_{pij}(A)\,\vert\,p\in\cP'(A),\,1\leq i\leq t_p(A)<j\leq n\}, \label{eq:char10}\\
&&M(B;{\bf e}_1,\ldots, {\bf e}_n)=\{\al_{pij}(B)\,\vert\,p\in\cP'(B),\,1\leq i\leq t_p(B)<j\leq n\}.\label{eq:char20}
\end{eqnarray}
We say that $T\in\GL_n(\bQ)$ satisfies the condition $(A,B,p)$, $p\in\cP'(B)$, if %\begin{eqnarray*}
%&& j-\text{th column }
%\left(\begin{matrix} \ga_{1j} & \cdots & \ga_{nj} \end{matrix}\right)\text{ of }T \text{ satisfies \eqref{eq:b}} \\ && \text{with }b_i=\ga_{ij},\,\, \al_{pij}=\al_{pij}(B)
%\text{ for any } j\in\{t_p+1,n\}.
%\end{eqnarray*}
\begin{eqnarray*}
&& j-\text{th column }
\left(\begin{matrix} \ga_{1j} & \cdots & \ga_{nj} \end{matrix}\right)\text{ of }T \text{ satisfies } \\ && \ga_{kj}-\sum_{i=1}^{t_p}\ga_{ij}\al_{pik}(B)\in\bZ_p  \text{ for any } k,j\in\{t_p+1,n\}.
\end{eqnarray*}

\begin{thm}\label{th:main1}
Let $A,B\in\M_n(\bZ)$ be non-singular and let $G_A,G_B$ have characteristics \eqref{eq:char10}, \eqref{eq:char20}, respectively. For $T\in\GL_n(\bQ)$ we have $T(G_A)=G_B$ if and only if 
\begin{eqnarray*}
\cP=\cP(A)&=&\cP(B),\\ 
\cP'=\cP'(A)&=&\cP'(B), \\ 
\caR=\caR(A)&=&\caR(B),\\ 
t_p(A)&=&t_p(B), \quad\forall p\in\cP,
\end{eqnarray*}  
$T\in\GL_n(\caR)$, $T(D_p(A))=D_p(B)$, and $T$ $($resp., $T^{-1})$ satisfies the condition $(A,B,p)$  
$($resp., $(B,A,p))$ for any $p\in\cP'$.
\end{thm}

\begin{proof}
By Corollary \ref{cor:loc}, $T(G_A)=G_B$ if and only if for any prime $p\in\bN$
%\bbe\label{eq:tt00}
$$
T(\overline{G}_{A,p})=\overline{G}_{B,p}.
$$
%\ee
In particular, using Proposition \ref{prop:1}, if $T(G_A)=G_B$, then 
$\cP(A)=\cP(B)$, $\cP'(A)=\cP'(B)$, $t_p(A)=t_p(B)$, and hence $\caR(A)=\caR(B)$. Also, $T\in\GL_n(\caR)$ by \cite[Lemma 3.4]{s}.

\sbr

By Lemma \ref{lem:gen-n} applied to $K=\bQ$,
\begin{eqnarray}%\label{eq:deomp}
\overline{G}_{A,p}&=&D_p(A)\oplus\Span_{\bZ_p}({\bf e}_{t_1+1},\ldots,{\bf e}_n),\label{eq:dec1}
\\
\overline{G}_{B,p}&=&D_p(B)\oplus\Span_{\bZ_p}({\bf e}_{t_2+1},\ldots,{\bf e}_n),\label{eq:dec2}
\end{eqnarray}
where $t_1=t_p(A)$, $t_2=t_p(B)$, and $D_p(A)\cong\bQ_p^{t_1}$, $D_p(B)\cong\bQ_p^{t_2}$ as $\bZ_p$-modules.
Therefore, $T$ defines a $\bZ_p$-module
isomorphism from 
$\overline{G}_{A,p}$ to $\overline{G}_{B,p}$ if and only if $t=t_1=t_2$, 
$T(D_p(A))=D_p(B)$, and with respect to the decompositions \eqref{eq:dec1} and \eqref{eq:dec2},  $T$ has the form
%\bbe\label{eq:want}
$$
\tilde T=\left(\begin{matrix}
T_1 & * \\
0 & T_2 
\end{matrix}
\right),\quad T_1\in\GL_{t}(\bQ_p),\,\,T_2\in\GL_{n-t}(\bZ_p).
$$
%\ee
Note that $T_2\in\GL_{n-t}(\bZ_p)$ if and only if 
$T{\bf e}_j\in\overline{G}_{B,p}$ and $T^{-1}{\bf e}_j\in\overline{G}_{A,p}$ for any $j\in\{t+1,\ldots, n\}$, which is equivalent to the conditions $(A,B,p)$, $(B,A,p)$ for columns of $T$, $T^{-1}$, respectively, by Lemma \ref{lem:gen-n-bel}.
\end{proof}
%\begin{rem}
%Mistake in \cite{gm}?
%\end{rem}

\section{Generalized eigenvectors}\label{s:eigen}
%\section{Divisible part of $\overline{G}_{A,p}$} 
Let $A,B\in\M_n(\bZ)$ be non-singular. 
Using Theorem \ref{th:main1}, one can already check whether $G_A\cong G_B$ and also find such isomorphisms if they exist. 
In this section, we make Theorem \ref{th:main1} even more practical  by 
describing the $\bZ_p$-divisible part $D_p(A)$ of $G_A\otimes_{\bZ}\bZ_p$ in terms of generalized eigenvectors of $A$.

\sbr 

Throughout the text, $\overline{\bQ}$ denotes a fixed algebraic closure of $\bQ$. 
Let $K\subset \overline{\bQ}$ be a finite extension of $\bQ$ that contains all the eigenvalues of $A$. Let $\cO_K$ denote the ring of integers of $K$.  Throughout the paper,  
$\la_1,\ldots,\la_n\in\cO_K$ denote (not necessarily distinct) eigenvalues of $A$ and 
$\{{\bf u}_1,\ldots,{\bf u}_n\}$ denotes a Jordan canonical basis of $A$. 
%corresponding to $\la_1,\ldots,\la_n$, respectively. 
Without loss of generality, we can assume that each
${\bf u}_i\in(\cO_K)^n$, $i=1,\ldots,n$. 
 For a prime $p\in\bN$ let $\fp$ be a prime ideal of $\cO_K$ above $p$ and let $X_{A,\fp}$ %(resp., $Y_p$) 
denote the span over $K$ of vectors in $\{{\bf u}_1,\ldots,{\bf u}_n\}$ 
%(resp., of $B$) 
corresponding to eigenvalues divisible by $\fp$. 
%Note that for any prime ideals $\fp_1$, $\fp_2$ of $\cO_K$ above $p$ there exists $\sg\in\Gal(K/\bQ)$ such that $\sg(\fp_1)=\fp_2$ and hence 
%$$
%\sg(X_{A,\fp_1})=X_{A,\fp_2}.
%$$ 
%In particular, the dimension of $X_{A,\fp_1}$ over $K$ does not depend on the choice of 
%$\fp$ above $p$. 
%Let $t_p(A)$ denote the multiplicity of zero in the reduction of the characteristic polynomial of $A$ modulo $p$.
Note that 
%if $A$ is diagonalizable ({\it e.g.}, if the characteristic polynomial $h_A$ of $A$ is irreducible), then
$$
\dim_K X_{A,\fp}=t_p(A),
$$
where $t_p=t_p(A)$ denotes the multiplicity of zero in the reduction 
$\bar{h}_A$ modulo $p$ of the characteristic polynomial $h_A$ of $A$, $0\leq t_p\leq n$. Indeed, 
$\dim_K X_{A,\fp}$ is the number of 
eigenvalues (with multiplicities) of $A$ divisible by $\fp$. One can write $h_A=(x-\la_1)\cdots (x-\la_n)$ over 
$\cO_K$. Considering the reduction $\bar{h}_A$ of $h_A$ modulo $\fp$, we see that the number of eigenvalues of $A$ divisible by $\fp$ is equal to the multiplicity 
$t_p$ of zero in $\bar{h}_A$. Equivalently, $X_{A,\fp}$ is generated over $K$ by generalized 
$\la$-eigenvectors of $A$ for any eigenvalue $\la$ of $A$ divisible by $\fp$.

\begin{lem}\label{l:div}
Let $A\in\M_n(\bZ)$ be non-singular.
%Assume the characteristic polynomial of $A$ is irreducible. 
%Assume $A$ is diagonalizable.
Let $p\in\bN$ be prime and let $\fp$ be a prime ideal of 
$\cO_K$ above $p$. Let $\cO_{\fp}$ denote the ring of integers of $K_{\fp}$, the completion of $K$ with respect to $\fp$. Then, considered as subsets of $K_{\fp}^n$,
$$
D_p(A)\otimes_{\bZ_p}\cO_{\fp}=X_{A,\fp}\otimes_K K_{\fp},
$$
{\it i.e.}, upon the extension of constants from $\bZ_p$ to $\cO_{\fp}$, the divisible part of
$\overline{G}_{A,p}$ is generated over $K_{\fp}$ by generalized eigenvectors of $A$ $($considered as elements of $K_{\fp}^n$ via the embedding $K\hrar K_{\fp})$ corresponding to eigenvalues divisible by $\fp$.
\end{lem}

\begin{proof}
Let $(\pi)\subset\cO_{\fp}$ denote the prime ideal of $\cO_{\fp}$. Via $K\hrar K_{\fp}$, we have 
$\la_1,\ldots,\la_n\in\cO_{\fp}$ and without loss of generality, we can assume 
$\la_1,\ldots,\la_{t_p}\in (\pi)$,  
$\la_{t_p+1},\ldots,\la_n\in (\cO_{\fp})^{\times}$. Thus,
$Y=X_{A,\fp}\otimes_K K_{\fp}$ is generated over $K_{\fp}$ by generalized eigenvectors of $A$ corresponding to $\la_i$, $i=1,\ldots,t_p$. Let
$$
Z=\overline{G}_{A,p}\otimes_{\bZ_p}\cO_{\fp}=
\left.\left\{A^{-k}{\bf x}\,\right\vert\, {\bf x}\in\cO_{\fp}^n,\,k\in\bZ \right\},\quad
\cO_{\fp}^n\subseteq Z\subseteq K_{\fp}^n.
$$
Using Lemma \ref{l:repr}, we have
$$
Z=\overline{G}_{A,p}\otimes_{\bZ_p}\cO_{\fp}=
\left(D_p(A)\otimes_{\bZ_p}\cO_{\fp}\right)\oplus 
\left(R_p(A)\otimes_{\bZ_p}\cO_{\fp}\right),
$$
where, as $\cO_{\fp}$-modules,
$$
D_p(A)\otimes_{\bZ_p}\cO_{\fp}\cong K_{\fp}^{t_p},\quad
R_p(A)\otimes_{\bZ_p}\cO_{\fp}\cong \cO_{\fp}^{n-t_p}.
$$
We first prove $Y\subseteq Z$, by showing $\Span_{K_{\fp}}({\bf u})\subseteq Z$ for any generalized eigenvector ${\bf u}$ corresponding to $\la_i$, $i=1,\ldots,t_p$. The proof is by induction on the rank of ${\bf u}$. Without loss of generality, we can assume 
${\bf u}\in \cO_{\fp}^{n}$. If $\rank {\bf u}=1$, then ${\bf u}$ is an eigenvector of $A$ corresponding to $\la_i$ and hence $\la_i^{-k}{\bf u}=A^{-k}{\bf u}\in Z$ for any $k\in\bZ$. Since $\la_i=\pi^{\alpha}\beta$ for $\alpha\in\bN$, $\be\in(\cO_{\fp})^{\times}$, and $Z$ is an $\cO_{\fp}$-module, we have
$\Span_{K_{\fp}}({\bf u})\subseteq Z$. Assume now $\rank {\bf u}=m$, $m>1$. Then, 
$(A-\la_i\operatorname{Id})^m{\bf u}={\bf 0}$, where $\operatorname{Id}$ denotes the $n\times n$-identity matrix, and ${\bf v}=(A-\la_i\operatorname{Id}){\bf u}$ is of rank $m-1$. By induction on $m$,
$\Span_{K_{\fp}}({\bf v})\subseteq Z$. We have ${\bf v}=A{\bf u}-\la_i{\bf u}$ and hence
\bbe\label{eq:tick}
\la_i^{-k}A^{-1}{\bf v}=\la_i^{-k}{\bf u}-\la_i^{-k+1}A^{-1}{\bf u},\quad\quad k\in\bZ.
\ee
From \eqref{eq:tick}, we can show $\la_i^{-k}{\bf u}\in Z$ by induction on $k\geq 0$. Indeed, for $k=0$, we have
${\bf u}\in Z$,
since 
${\bf u}\in\cO_{\fp}^n$. Assume $\la_i^{-(k-1)}{\bf u}\in Z$. Then
$A^{-1}(\la_i^{-(k-1)}{\bf u})=\la_i^{-k+1}A^{-1}{\bf u}\in Z$, since $Z$ is $A^{-1}$-invariant.
Analogously, $\la_i^{-k}A^{-1}{\bf v}\in Z$, since $\la_i^{-k}{\bf v}\in Z$ by induction on the rank.
Thus, 
$\la_i^{-k}{\bf u}\in Z$ by \eqref{eq:tick}. As before, it shows $\Span_{K_{\fp}}({\bf u})\subseteq Z$.
Here ${\bf u}$ is a generalized eigenvector of an arbitrary rank corresponding to an eigenvalue of $A$ divisible by $\fp$ and hence $Y\subseteq Z$. Finally, since 
$Y$ is
a divisible $\cO_{\fp}$-module, it is contained inside the divisible part of $Z$, {\it i.e.}, $Y\subseteq D_p(A)\otimes_{\bZ_p}\cO_{\fp}$. 
Since both have the same dimension $t_p$ over $K_{\fp}$, they coincide and the claim follows.
\end{proof}
\begin{rem}
Note that we cannot claim that the reduced part 
$R_p(A)\otimes_{\bZ_p}\cO_{\fp}$ of 
$\overline{G}_{A,p}\otimes_{\bZ_p}\cO_{\fp}$ is generated by generalized eigenvectors of $A$ over $\cO_{\fp}$, since in general, $\{ {\bf u}_1,\ldots,{\bf u}_{n}\}$ is not 
a free basis of $\cO_{\fp}^n$. Equivalently, the matrix 
$\left(\begin{matrix}
{\bf u}_{1}  & \ldots & {\bf u}_{n}
\end{matrix}
\right)$ might not be in $\GL_n(\cO_{\fp})$.
\end{rem}

%\section{Isomorphism classes of $G_A$} In this section we prove a characterization for $G_A\cong G_B$, where $A,B\in\M_n(\bZ)$ are non-singular. %with irreducible characteristic polynomials. 

Combining Lemma \ref{l:div} with Theorem \ref{th:main1}, we get a criterion for $G_A\cong G_B$ in terms of generalized eigenvectors of $A$ and $B$. 

\begin{thm}\label{th:prelim0}
Let $A,B\in\M_n(\bZ)$ be non-singular, let $K\subset \overline{\bQ}$ be any finite  extension of $\bQ$ that contains the eigenvalues of both $A$ and $B$, and let $G_A$, $G_B$ have characteristics 
\eqref{eq:char10}, \eqref{eq:char20}, respectively. 
%Assume the characteristic polynomials of 
%$\chi_A\in\bZ[t]$ of 
%$A$, $B$ are irreducible. 
For $T\in\GL_n(\bQ)$ we have $T(G_A)=G_B$ if and only if 
\begin{eqnarray*}
\cP=\cP(A)&=&\cP(B),\\ 
\cP'=\cP'(A)&=&\cP'(B), \\ 
\caR=\caR(A)&=&\caR(B),\\ 
t_p(A)&=&t_p(B), \quad\forall p\in\cP,
\end{eqnarray*}  
$T\in\GL_n(\caR)$, 
%there exist eigenvalues $\la,\mu\in K$ 
%corresponding to eigenvectors ${\bf u},{\bf v}\in K ^n$ 
%of $A,B$, respectively, such that $\la$, $\mu$ have the same prime ideal divisors in the ring of integers of $K$, 
for any $p\in\cP'$ and a prime ideal $\fp$ of $\cO_K$ above $p$
%and eigenvalues $\la,\mu\in \overline{\bQ}$ of $A,B$, respectively, 
we have
\begin{eqnarray*}
&& T(X_{A,\fp})=X_{B,\fp},%\label{eq:class0}
\end{eqnarray*}
and $T$ $($resp., $T^{-1})$ satisfies the condition $(A,B,p)$  
$($resp., $(B,A,p))$ for any $p\in\cP'$.
% the restriction $T_p$ of $T$ with respect to the bases ${\bf f}_{t_p+1},\ldots, {\bf f}_{n}$ and ${\bf g}_{t_p+1},\ldots, {\bf g}_{n}$ is a $\bZ_p$-module isomorphism.
%%$$
%T_p:\Span_{\bZ_p}({\bf f}_{t_p+1},\ldots, {\bf f}_{n})\rar
%\overline{G}_{B,p}/\Span_{\bZ_p}({\bf g}_{1},\ldots, {\bf g}_{t_p})
%\cong\Span_{\bZ_p}({\bf g}_{t_p+1},\ldots, {\bf g}_n)
%$$
\end{thm}

\begin{proof}
%By Corollary \ref{cor:loc}, $T(G_A)=G_B$ if and only if for any prime $p\in\bZ$
%\bbe\label{eq:tt00}
%T(\overline{G}_{A,p})=\overline{G}_{B,p}.
%\ee
%Let $\fp$ denote a prime ideal of $\cO_K$ lying above $p$ and let $\cO_{\fp}$ denote the ring of integers of $K_{\fp}$, the completion of $K$ with respect to 
%$\fp$. %Using the action of $\Gal(K_{\fp}/\bQ_p)$, 
%One can show that \eqref{eq:tt00} holds if and only if
%\bbe\label{eq:5670}
%T(\overline{G}_{A,p}\otimes_{\bZ_p}\cO_{\fp})=\overline{G}_{B,p}\otimes_{\bZ_p}\cO_{\fp},
%\ee
We have  
$T(D_p(A))=D_p(B)$ if and only if 
$T(D_p(A)\otimes_{\bZ_p}\cO_{\fp})=D_p(B)\otimes_{\bZ_p}\cO_{\fp}$, since $T$ is defined over $\bQ$. By Lemma \ref{l:div}, 
$D_p(A)\otimes_{\bZ_p}\cO_{\fp}=X_{A,\fp}\otimes_K K_{\fp}$ for any prime ideal $\fp$ of $\cO_K$ above $p$. Finally, $T(X_{A,\fp}\otimes_K K_{\fp})=X_{B,\fp}\otimes_K K_{\fp}$ if and only if $T(X_{A,\fp})=X_{B,\fp}$, since $T$ is defined over $\bQ$. Thus, the theorem follows from Theorem \ref{th:main1}.
\end{proof}

\begin{rem}\label{r:char}
We find Theorem \ref{th:prelim0} more practical than Theorem \ref{th:main1}. The difference between the two is that to find a characteristic of $G_A$ using Theorem \ref{th:main1}, for each $p$ one finds
a possibly different matrix $W_p$ and then modifies the rows according to the procedure described in Lemma \ref{l:vect} to get a basis $\{{\bf f}_1,\ldots,{\bf f}_n\}$ (see Remark \ref{rem:base}). Whereas, in 
Theorem \ref{th:prelim0}, we can start with a Jordan canonical basis of $A$ (which does not depend on $p$) and then modify it using the same procedure (see Example \ref{ex:aaat} below). 
By Lemma \ref{lem:gen-n} (and, possibly, Lemma \ref{l:stand}), up to an isomorphism of $G_A$, both ways produce the same characteristic. 

\end{rem}

\section{Reducible characteristic polynomials}
Let $A,B\in\M_n(\bZ)$ be non-singular with $G_A$, $G_B$ defined by \eqref{eq:opr}.
In this section we explore necessary conditions for $G_A\cong G_B$, when at least one of the characteristic polynomials of $A$, $B$ is reducible in $\bZ[t]$.
\subsection{Irreducible isomorphisms} We start by introducing the notion of  {\it an irreducible isomorphism} between $G_A$ and $G_B$. 
 Let $K\subset\overline{\bQ}$ denote a finite Galois extension of $\bQ$ that contains all the eigenvalues of $A$ and $B$ and let 
$G=\Gal(K/\bQ)$. For an eigenvalue $\la\in K$ of $A$ let $K(A,\la)$ denote the generalized $\la$-eigenspace of $A$. By definition, $K(A,\la)$ is generated over $K$ by all generalized eigenvectors of $A$ corresponding to $\la$ or, equivalently,   by vectors in  a Jordan canonical basis of $A$ corresponding to $\la$. Let $h_A\in\bZ[t]$ denote the characteristic polynomial of $A$. Assume 
$h_A=fg$ for non-constant $f,g\in\bZ[t]$. By Theorem \ref{th:trig} below, there exists 
$S\in\GL_n(\bZ)$ such that
%\bbe\label{eq:apt}
$$
SAS^{-1}=\left(\begin{matrix}
A' & * \\
0 & A''
\end{matrix}
\right),
$$
%\ee
where $A',A''$ are matrices with integer coefficients of appropriate sizes such that the characteristic polynomial of $A'$ (resp., $A''$) is $f$ (resp., $g$). 
We have a natural embedding $G_{A'}\hrar G_{SAS^{-1}}$ induced by ${\bf x}\mapsto \left(\begin{matrix}
{\bf x}  &  {\bf 0}\end{matrix}\right)$, where ${\bf x}\in\bQ^{n_1}$, $n_1=\deg f$, and ${\bf 0}$ is the zero vector in $\bQ^{n-n_1}$. There is an exact sequence
\bbe\label{eq:seq10}
0\rar G_{A'}\rar G_{A}\rar G_{A''}\rar 0,
\ee
since $S(G_A)=G_{SAS^{-1}}$. We denote $G_{A'}=G_f$, $G_{A''}=G_g$. 
\begin{defin}\label{def:1}
We say that an isomorphism $T:G_A\rar G_B$ is {\em reducible} if there exist $S,L\in\GL_n(\bZ)$ and non-constant $f,g,f',g'\in\bZ[t]$ such that $h_A=fg$, $h_B=f'g'$, 
$$
SAS^{-1}=\left(\begin{matrix}
A' & * \\
0 & A''
\end{matrix}
\right),\quad
LBL^{-1}=\left(\begin{matrix}
B' & * \\
0 & B''
\end{matrix}
\right),
$$
$h_{A'}=f$, $h_{B'}=f'$, $\deg f=\deg f'$, and $LTS^{-1}(G_f)=G_{f'}$.
Otherwise, we say that $T$ is {\em irreducible}.
\end{defin}

Clearly, if the characteristic polynomial of $A$ or $B$ is irreducible, then an isomorphism $T:G_A\rar G_B$ is irreducible. The converse is not true in general. For instance,
$$
{A}=\left(\begin{matrix}
2 & 1   \\
0 & 2  
\end{matrix}
\right),\quad
{B}=\left(\begin{matrix}
4 & 1   \\
0 & 4  
\end{matrix}
\right),
\quad
{T}=\left(\begin{matrix}
1 & 2   \\
1 & 4  
\end{matrix}
\right),
$$
where both characteristic polynomials $h_A$, $h_B$ are reducible, but $T:G_A\rar G_B$ is an irreducible isomorphism.
Indeed, any $S,L,T\in\GL_2(\bQ)$ satisfying the conditions in Definition \ref{def:1} have to be upper-triangular. However, 
for $\caR=\bZ\left[\frac{1}{2}\right]$ any $T\in\GL_2(\caR)$  is an isomorphism between $G_A$ and $G_B$ by Corollary \ref{cor:loc} and Proposition \ref{prop:1}.

\sbr

Note that $LTS^{-1}(G_f)=G_{f'}$ if and only if
\bbe\label{eq:jord}
T\left(\sum_{\la} K(A,\la)\right)=\sum_{\mu} K(B,\mu),
\ee
where $\la\in\overline{\bQ}$ (resp., $\mu\in\overline{\bQ}$) runs through all the roots of $f$ 
(resp., $f'$).
Also, $LTS^{-1}(G_f)=G_{f'}$ implies $LTS^{-1}(G_g)=G_{g'}$. Thus, if $T$ is reducible, then
$G_f\cong G_{f'}$, $G_g\cong G_{g'}$. In other words, if $h_A=fg$ and there is a reducible isomorphism 
$G_A\cong G_B$, 
then $G_f\cong G_{f'}$, $G_g\cong G_{g'}$ for some $f',g'\in\bZ[t]$ such that $h_B=f'g'$. The converse is not true in general. 
%In particular, whether the sequence \eqref{eq:seq1} splits. In general, if 
%%\bbe\label{eq:kkkk111}
%$$
%{A}=\left(\begin{matrix}
%A' & *   \\
%0 & A''  
%\end{matrix}
%\right),
%$$
%%\ee
%where $A',A''$ are matrices with integer coefficients, then the exact sequence 
%\bbe\label{eq:seq22}
%0\rar G_{A'}\rar G_A\rar G_{A''}\rar 0
%\ee
%does not split.
\begin{example}
Let
$$
{A}=\left(\begin{matrix}
2 & 0   \\
0 & 5  
\end{matrix}
\right),\quad
{B}=\left(\begin{matrix}
2 & 4   \\
0 & 5  
\end{matrix}
\right).
$$
Here, in the notation of Definition \ref{def:1}, $f(t)=f'(t)=t-2$, $g(t)=g'(t)=t-5$, 
$G_A\cong G_f\oplus G_{g}$, where $G_f=\{\frac{k}{2^n}\,\vert\, k,n\in\bZ\}$, 
$G_g=\{\frac{k}{5^n}\,\vert\, k,n\in\bZ\}$. Using Theorem \ref{th:main1} together with Lemma \ref{l:div}, one can show $G_A\not\cong G_B$, hence the sequence 
$$
0\rar G_{f'}\rar G_B\rar G_{g'}\rar 0
$$ 
does not split. This is also an example when $G_{f}\cong G_{f'}$, $G_{g}\cong G_{g'}$, but $G_A\not\cong G_B$. 
%Thus, the assumption in Remark \ref{rem:all} is necessary but not sufficient.
\end{example}

\subsection{Splitting sequences}
There is a case, however, when sequence \eqref{eq:seq10} splits, namely, when $\det A''=\pm 1$. Then, $G_{A''}=\bZ^{k}$ is a free $\bZ$-module, $A''\in\M_k(\bZ)$. More precisely, let $A\in\M_n(\bZ)$ be non-singular with characteristic polynomial $h_A\in\bZ[t]$. Let $h_A=fg$,
where $f,g\in\bZ[t]$ are non-constant, $f=f_1f_2\cdots f_s$, 
$f_i(0)\ne\pm 1$ for each irreducible component $f_i\in\bZ[t]$ of $f$, $1\leq i\leq s$,
$g(0)=\pm 1$. Then $G_{g}=\bZ^{k}$, $k=k(A)=\deg g$, %Let $n_1=\deg g$, $n_2=\deg f$. 
%By above,
%Theorem \ref{th:trig} below, there exists $S\in\GL_n(\bZ)$ such that
%\bbe\label{eq:kkkk11}
%SAS^{-1}=\left(\begin{matrix}
%A' & *   \\
%0 & A''  
%\end{matrix}
%\right),
%\ee
%where $A', A''$ are matrices with integer coefficients, $A'$ (resp., $A''$) has characteristic polynomial $f$ (resp., $g$).
%Then $S(G_A)=G_{SAS^{-1}}$, $G_{A''}=\bZ^{k}$, $k=k(A)=\deg g$, 
and hence the sequence
%\bbe\label{eq:spl}
$$
0\rar G_{f}\rar G_{A}\rar G_{g}\rar 0
%\ee
$$
splits, {\it i.e.}, 
\bbe\label{eq:sas1}
G_{A}\cong G_{f}\oplus\bZ^{k(A)}.
\ee

\begin{lem}\label{l:split}
Let $A,B\in\M_n(\bZ)$ be non-singular with corresponding characteristic polynomials $h_A,h_B\in\bZ[t]$. Then 
$$
G_A\cong G_B\iff  k(A)=k(B),\quad G_{f}\cong G_{f'},
$$
where $h_B=f'g'$, $r(0)\ne\pm 1$ for each irreducible component $r\in\bZ[t]$ of $f'$, and 
$g'(0)=\pm 1$.
\end{lem}

\begin{proof}
Clearly, the conditions are sufficient by \eqref{eq:sas1}. We now show that they are necessary.
Assume $G_A\cong G_B$. By \eqref{eq:sas1}, without loss of generality, we can assume that 
$$G_A=G_{f}\oplus\bZ^{k(A)},\quad G_B=G_{f'}\oplus\bZ^{k(B)}.$$ 
By Lemma 
\ref{l:dense} below, $G_{f}$ is dense in $\bQ^{n-k(A)}$. Therefore, the closure 
$\overline{G}_{A}$ of $G_{A}$ in $\bQ^{n}$ with its usual topology is 
$$
\overline{G}_A=\bQ^{n-k(A)}\oplus\bZ^{k(A)}
$$ 
and, analogously, for $B$
$$
\overline{G}_B=\bQ^{n-k(B)}\oplus\bZ^{k(B)}.
$$ 
An isomorphism between $G_A$ and $G_B$ is induced by a linear isomorphism 
$T\in\GL_n(\bQ)$ of
$\bQ^n$ \cite[Lemma 3.1]{s} such that $T(G_A)=G_B$. Thus, $T(\overline{G}_A)=\overline{G}_B$, hence $k(A)=k(B)$,
$T(\bQ^{n-k(A)})=\bQ^{n-k(B)}$, and therefore $T({G}_{f})={G}_{f'}$.
\end{proof}

\begin{rem}\label{rem:ones}
By Lemma \ref{l:split}, without loss of generality, for the rest of the section we can assume that
$r(0)\ne\pm 1$ for any irreducible component $r\in\bZ[t]$ of $h_A$, and the same holds for $h_B$.
\end{rem}

\subsection{Properties of irreducible isomorphisms}
We now explore necessary conditions for an isomorphism between $G_A$ and $G_B$ to be irreducible. 
For any 
$p\in\cP'(A)$ let $\tilde{h},h_{A,p}\in\bZ[t]$  be such that 
$h_A=\tilde{h}h_{A,p}$, $p$ does not divide $\tilde{h}(0)$, and $p$ divides $r(0)$ for any irreducible component $r\in\bZ[t]$ of $h_{A,p}$. Also, let $S_{A,p}$ denote the set of distinct roots of $h_{A,p}$ (not counting multiplicities). For a prime ideal $\fp$ of the ring of integers $\cO_K$ of $K$ above $p$, let
$$
X_{A,p}=\sum_{\sg\in G}\sg(X_{A,\fp})=\sum_{\sg\in G}X_{A,\sg(\fp)},
$$
where the second equality holds, since $A$ is defined over $\bQ$, $G=\Gal(K/\bQ)$, and $X_{A,\fp}$ is defined in Section \ref{s:eigen}.  Equivalently, 
$$
X_{A,p}=\sum_{\la\in S_{A,p}}K(A,\la),\quad S_{A,p}=\{\la\in\cO_K\,\vert\,h_{A,p}(\la)=0\},
$$
since $G$ acts transitively on the roots of an irreducible component $r\in\bZ[t]$ of $h_A$. Note that 
$$
\dim X_{A,p}=\deg h_{A,p},\quad \sg(X_{A,p})=X_{A,p}\text{ for any }\sg\in G.
$$
Moreover, for $p_1,\ldots,p_k\in\cP'(A)$ denote recursively
$$
X_{A,p_1\cdots p_k}=X_{A,p_1\cdots \,p_{k-1}}\cap X_{A,p_k}=
\sum_{\la\in S_{A,p_1}\cap\,\cdots\,\cap\, S_{A,p_k}}K(A,\la),
$$ 
where the second equality holds, since generalized eigenvectors corresponding to distinct eigenvalues are linearly independent. We write $h_A=h_1\cdots h_s$, where each 
$h_i=r_i^{u_i}$, $u_i\in\bN$, $r_i\in\bZ[t]$ is irreducible, and $h_i$, $h_j$ have no common roots in $\overline{\bQ}$ for $i\ne j$. In this notation,
$$
h_{A,p}=\prod_{p\vert h_i(0)}h_i,\quad h_{A,p_1\cdots \,p_k}=\prod_{p_1\cdots \,p_k\vert h_i(0)}h_i,
$$
where $p_1,\ldots,p_k$ are assumed to be distinct. Then, 
$\dim X_{A,p_1\cdots\,p_k}=\deg h_{A,p_1\cdots\,p_k}$. We now assume $B\in\M_n(\bZ)$ is non-singular and $T(G_A)=G_B$ for some $T\in\GL_n(\bQ)$. Then, 
by Theorem \ref{th:prelim0}, we have $\cP'=\cP'(A)=\cP'(B)$ and $T(X_{A,\fp})=X_{B,\fp}$. Since 
$T, A,B$ are all defined over $\bQ$, for any $\sg\in G$ we have 
$$ 
T(X_{A,\sg(\fp)})=T\sg(X_{A,\fp})=\sg(T(X_{A,\fp}))=\sg(X_{B,\fp})=X_{B,\sg(\fp)},
$$
and hence $T(X_{A,p})=X_{B,p}$. This implies the following lemma.
\begin{lem}\label{l:degrees}
Let $A,B\in\M_n(\bZ)$ be non-singular and let $T(G_A)=G_B$, $T\in\GL_n(\bQ)$. 
Then
$\cP'=\cP'(A)=\cP'(B)$ and for any $k\in\bN$ with distinct $p_1,\ldots,p_k\in\cP'$, 
$$
T(X_{A,p_1\cdots\,p_k})=X_{B,p_1\cdots\,p_k}.
$$ 
In particular, 
$$
\deg h_{A,p_1\cdots\,p_k}=\deg h_{B,p_1\cdots\,p_k}.
$$
\end{lem}

\begin{example}
Let $A,B\in\M_5(\bZ)$ be non-singular with characteristic polynomials 
$$
h_A=(t^2+t+2)(t^3+t+6),\quad h_B=(t^2+4)(t^3+t+3).
$$
Then $\cP'=\cP'(A)=\cP'(B)=\{2,3\}$, $t_2(A)=t_2(B)=2$, 
$t_3(A)=t_3(B)=1$. However,  
$$
h_{A,2}=h_A,\quad h_{B,2}=t^2+4,
$$
so that $\deg h_{A,2}\ne\deg h_{B,2}$ and hence $G_A\not\cong G_B$ by Lemma \ref{l:degrees}.
\end{example}

\begin{cor}
If an isomorphism $T:G_A\rar G_B$ is irreducible, then for all the irreducible components $f_1,\ldots,f_k\in\bZ[t]$ $($resp., $g_1,\ldots,g_s\in\bZ[t])$ of the characteristic polynomial of $A$ 
$($resp., of $B)$, all $f_1(0)$, $\ldots,f_k(0)$ $($resp., $g_1(0),\ldots,g_s(0))$ have the same prime divisors $($in $\bZ)$.
\end{cor}

\begin{proof}
Assume $T(G_A)=G_B$, $T\in\GL_n(\bQ)$. By Theorem \ref{th:prelim0},  
$\cP'=\cP'(A)=\cP'(B)$. 
In the above notation, for $p\in\cP'$ and the characteristic polynomial 
$h_A$ (resp., $h_B$) of $A$ (resp., $B$) let $\tilde{h},\hat{h},h_{A,p}, h_{B,p}\in\bZ[t]$ be such that 
$h_A=\tilde{h}h_{A,p}$, $h_B=\hat{h}h_{B,p}$, $p$ does not divide $\tilde{h}(0)\hat{h}(0)$, and $p$ divides $r(0)$ for any irreducible component $r\in\bZ[t]$ of $h_{A,p}h_{B,p}$. It follows from Lemma \ref{l:degrees}, \eqref{eq:jord}, and the paragraph preceding Definition \ref{def:1} that $LTS^{-1}(G_{h_{A,p}})=G_{h_{B,p}}$ for some $L,S\in\GL_n(\bZ)$. Since  
$T$ is irreducible, $\tilde h$ is constant. Since $p\in\cP'$ is arbitrary,
we conclude that for all the irreducible components $f_1,\ldots,f_k\in\bZ[t]$ of $h_A$, $f_1(0),\ldots,f_k(0)$ have the same prime divisors (in $\bZ$). By symmetry, the same holds for $B$.
\end{proof}

\subsection{Galois action}\label{ss:galois}
We explore the action of the Galois group $\Gal(K/\bQ)$ on eigenvalues of non-singular $A,B\in\M_n(\bZ)$ when $G_A\cong G_B$. Let $A,B$ have characteristic polynomials 
$h_A=h_1^{\al_1}\cdots h_k^{\al_k}$, 
$h_B=r_1^{\be_1}\cdots r_s^{\be_s}$, respectively, where $\al_1,\ldots,\al_k,\be_1,\ldots,\be_s\in\bN$, and $h_1,\ldots,h_k\in\bZ[t]$ (resp., $r_1,\ldots,r_s\in\bZ[t]$) are distinct and  irreducible. Let $K\subset\overline{\bQ}$ be a finite Galois extension of $\bQ$ that contains all the eigenvalues of $A$ and $B$. Let $\Sg\subset K$ (resp., $\Sg'\subset K$) denote the set of all distinct eigenvalues of $A$ (resp., $B$) with cardinality denoted by $|\Sg|$, and let 
$\Sg=\Sg_1\sqcup\cdots\sqcup\Sg_k$ (resp., $\Sg'=\Sg'_1\sqcup\cdots\sqcup\Sg'_s$), where each $\Sg_i$ (resp., $\Sg'_j$) is the set of all (distinct) roots of $h_i$ (resp., $r_j$), $i\in\{1,\ldots,k\}$, $j\in\{1,\ldots,s\}$. Thus,
$$
n=\sum_{i=1}^k\al_i|\Sg_i|=\sum_{j=1}^s\be_j|\Sg'_j|,\quad n_i(A)=|\Sg_i|,\quad n_j(B)=|\Sg'_j|,
$$
where $n_i(A)$ (resp., $n_j(B)$) is the number of distinct roots of $h_i$ (resp., $r_j$).

\sbr

Let $T:G_A\rar G_B$ be an isomorphism.
By  Theorem \ref{th:prelim0}, 
$\caR=\caR(A)=\caR(B)$, $\cP=\cP(A)=\cP(B)$,
$\cP'=\cP'(A)=\cP'(B)$, and $t_p=t_p(A)=t_p(B)$ for any prime $p\in\cP$. By assumption, $\cP'\ne\emptyset$ and for any $p\in\cP'$ we have $1\leq t_p\leq n-1$. For a subset 
$M$ of $\Sg$ (resp., $M'$ of $\Sg'$) we denote
\begin{eqnarray*}
U_M&=&\bigoplus_{\la\in M}K(A,\la),\quad M=M_1\sqcup\cdots\sqcup M_k,\\
V_{M'}&=&\bigoplus_{\mu\in M'}K(B,\mu),\quad M'=M'_1\sqcup\cdots\sqcup M'_s,
\end{eqnarray*}
where each $M_i$ (resp., $M'_j$) is a subset of $\Sg_i$ (resp., $\Sg'_j$), and 
$K(A,\la)$ (resp., $K(B,\mu)$) denotes the generalized $\la$-eigenspace of $A$ (resp., generalized $\mu$-eigenspace of $B$). Denote
$$
||M||=\sum_{i=1}^k\al_i|M_i|,\quad ||M'||=\sum_{j=1}^s\be_j|M'_j|.
$$
By Theorem \ref{th:prelim0}, we have $T(X_{A,\fp})=X_{B,\fp}$ for a prime ideal $\fp$ of $\cO_K$ above $p$, {\it i.e.,} in the above notation there exist $M\subset\Sg$, $M'\subset\Sg'$ such that 
\bbe\label{eq:tum1}
T(U_M)=V_{M'},\quad t_p=||M||=||M'||,\quad
t_{p,i}(A)=|M_i|,\quad t_{p,j}(B)=|M'_j|.
\ee
Here, $t_{p,i}(A)$ (resp., $t_{p,j}(B)$) is the number of distinct roots of $h_i$ (resp., $r_j$) divisible by $\fp$. Equivalently, 
$t_{p,i}(A)$ (resp., $t_{p,j}(B)$) is the multiplicity of zero in the reduction of $h_i$ (resp., $r_j$) modulo $p$. 
\begin{lem}\label{l:divide}
Assume $T:G_A\rar G_B$ is an irreducible isomorphism. 
Let $S\subset\Sg$ be a non-empty
subset of $\Sg$ of the smallest cardinality with the property that there exists $S'\subset\Sg'$ with
$$
T(U_S)=V_{S'},\quad S=S_1\sqcup\cdots\sqcup S_k,\quad S'=S'_1\sqcup\cdots\sqcup S'_s,
$$
where each $S_i$ $($resp., $S'_j)$ is a subset of $\Sg_i$ $($resp., $\Sg'_j)$. Then, $S_i\ne\emptyset$, $S'_j\ne\emptyset$ for any 
$i\in\{1,\ldots,k\}$, $j\in\{1,\ldots,s\}$.
Moreover, $||S||=||S'||$, for any $i$,  $p\in\cP'$,
\begin{enumerate}[$(a)$]
\item $||S||\text{ divides }n,\, t_p$, 
\item $|S_i|\text{ divides }n_i(A),\, t_{p,i}(A)$,
\item $\frac{n_i(A)}{|S_i|}=\frac{n}{||S||}$,
%\item $n_i(A)||S||=n|S_i|$,
\item $\frac{t_{p,i}(A)}{|S_i|}=\frac{t_p}{||S||}$,
%\item $t_{p,i}(A)||S||=t_p|S_i|$,
\end{enumerate}
and, similarly, for $B$. 
\end{lem}
\begin{proof}
By \eqref{eq:tum1}, $S$ exists and $1\leq ||S||<n$. Assume $T$ is irreducible and there exists $S_i=\emptyset$, 
{\it e.g.}, $S_1=\cdots=S_l=\emptyset$, $S_{l+1},\ldots,S_k$ are non-empty, $l\in\bN$, $1\leq l\leq k-1$, 
$f=h_{l+1}^{\al_{l+1}}\cdots h_k^{\al_k}$, $J=\{j\,\vert\, S'_j\ne\emptyset \}$, $J\ne\emptyset$, and $f'=\prod_{j\in J}r_j^{\be_j}$. From the definition of $S,S'$, we have
\bbe\label{eq:summ}
T\left(\bigoplus_{\la\in S} K(A,\la)\right)=\bigoplus_{\mu\in S'} K(B,\mu).
\ee 
By applying any $\sg\in\Gal(K/\bQ)$ to \eqref{eq:summ} and using the transitivity of the Galois action on roots of irreducible polynomials with rational coefficients, we see that
$$
T\left(\bigoplus_{\la\in\{\text{roots of }f\}} K(A,\la)\right)=\bigoplus_{\mu\in\{\text{roots of }f'\}} K(B,\mu).
$$
By the dimension count, this implies $\deg f'=\deg f<n$ and \eqref{eq:jord} holds. This contradicts the assumption that $T$ is irreducible.
%
%then it follows from \eqref{eq:jord} and the paragraph preceding Definition \ref{def:1} that there exist 
%$L,S\in\GL_n(\bZ)$ and $r\in\bZ[t]$ that divides $h_B$, $\deg r<\deg h_B$, such that
%$LTS^{-1}(G_{q})=G_{r}$ for $q=h_2^{\al_2}\cdots h_k^{\al_k}$. This contradicts the assumption that $T$ is irreducible. 
Thus, all $S_i\ne\emptyset$ and, analogously, all $S'_j\ne\emptyset$.

\sbr

The Galois group $G=\Gal(K/\bQ)$ acts on 
$\Sg$ by acting on each $\Sg_i$, {\it i.e.}, $\sg(\Sg_i)=\Sg_i$ for any $i\in\{1,\ldots,k\}$, 
$\sg\in G$. 
Note that for any $P,R\subseteq\Sg$, $P',R'\subseteq\Sg'$ and $\sg\in G$ we have
\begin{alignat}{2}
U_P\cap U_R&= U_{P\,\cap\, R}, & \quad  V_{P'}\cap V_{R'} &=V_{P'\,\cap\, R'}, \label{eq:syst11} \\
\sg(U_P)&= U_{\sg(P)}, & \quad  \sg(V_{P'})& =V_{\sg(P')}.\label{eq:syst21} 
\end{alignat}
Let $N\subseteq\Sg$, $N'\subseteq\Sg'$ satisfy 
\bbe\label{eq:tun1}
T(U_N)=V_{N'}.
\ee 
Let $\sg\in G$ be arbitrary. Applying $\sg$ to \eqref{eq:tun1} and using properties
\eqref{eq:syst11}, \eqref{eq:syst21}, we have $T(U_{\sg(N)})=V_{\sg(N')}$, since
$T\in\GL_n(\bQ)$.
Hence,
$T(U_{S\,\cap\,\sg(N)})=V_{S'\,\cap\,\sg(N')}$.
Since $S$ is the smallest with this property, either $S\,\cap\,\sg(N)=S$ or
$S\,\cap\,\sg(N)=\emptyset$. Equivalently, $\sg(S)\,\cap\,N=\sg(S)$ or
$\sg(S)\,\cap\,N=\emptyset$. In particular, taking $N=\tau(S)$
for an arbitrary $\tau\in G$, either $\sg(S)=\tau(S)$ or
$\sg(S)\,\cap\,\tau(S)=\emptyset$. Let 
$$
S=S_1\sqcup\cdots\sqcup S_k,\quad N=N_1\sqcup\cdots\sqcup N_k,\quad \forall S_i,N_i\subseteq\Sg_i,
$$
$i\in\{1,\ldots,k\}$. 
Then for any $\sg\in G$, we have either $\sg(S_i)\,\cap\,N_i=\sg(S_i)$ for all $i$ 
or $\sg(S_i)\,\cap\,N_i=\emptyset$ for all $i$. Analogously, for any $\sg,\tau\in G$, we have either 
$\sg(S_i)=\tau(S_i)$ for all $i$ 
or $\sg(S_i)\,\cap\,\tau(S_i)=\emptyset$ for all $i$. Moreover, since each
$h_i$  is irreducible, $G$ acts transitively on $\Sg_i$. 
This implies that each $N_i$ is a disjoint union of orbits $\sg(S_i)$ of $S_i$, $\sg\in G$ and, furthermore, there exists a subset $H\subseteq G$ depending on $N$ such that 
\bbe\label{eq:yay}
N_i=\bigsqcup_{\sg\in H} \sg(S_i),\quad |N_i|=|H|\cdot |S_i|\,\,\text{ for all }\,\, i.
\ee
Clearly, \eqref{eq:tun1} holds for $N=\Sg$ and also for $N=M$ by \eqref{eq:tum1}. Thus, by \eqref{eq:yay}, there exists $H_1,H_2\subseteq G$ such that
\begin{eqnarray*}
n_i(A)=|H_1||S_i|,\quad n&=&\sum_{i=1}^k\al_i|\Sg_i|=|H_1|\sum_{i=1}^k\al_i |S_i|=|H_1|\cdot ||S||,\\
t_{p,i}(A)=|H_2||S_i|,\quad t_p&=&\sum_{i=1}^k\al_i| M_i|=|H_2|\sum_{i=1}^k\al_i |S_i|=|H_2|\cdot ||S||.
\end{eqnarray*}
Hence, (a), (b), (c), and (d) hold. 
By symmetry, we have analogous formulas for $B$. 
\end{proof}
We now use Lemma \ref{l:divide} in a special case when the greatest common divisor $(n,t_p)$ of $n$ and $t_p$ is one,
{\it e.g.}, when
$t_p=1$, or $t_p=n-1$, or $n$ is prime. The conclusion is that an irreducible isomorphism $T$ between $G_A$, $G_B$ implies that both
 characteristic polynomials $h_A$, $h_B$ are irreducible and $T$ takes any eigenvector of $A$ to an eigenvector of $B$. 

\begin{prop}\label{pr:main31}
Let $A,B\in\M_n(\bZ)$ be non-singular. Assume there exists a prime $p\in\cP'(A)$ with 
$(n,t_p(A))=1$. 
%Let $K\subset \overline{\bQ}$ be a finite extension of $\bQ$ that contains the eigenvalues of both $A$ and $B$.
If $T\in\GL_n(\bQ)$ is an irreducible isomorphism from $G_A$ to $G_B$, then both $h_A, h_B$ are irreducible in $\bZ[t]$,
and there exist eigenvalues $\la,\mu\in\overline{\bQ}$ of $A,B$, respectively, such that $K=\bQ(\la)=\bQ(\mu)$. Moreover, 
$\la$ and $\mu$ have the same prime ideal divisors in $\cO_K$, and for an eigenvector ${\bf u}\in({\overline{\bQ}})^n$ of $A$, $T({\bf u})$ is an eigenvector of $B$.
%
%eigenvectors ${\bf u},{\bf v}\in K ^n$
%corresponding to eigenvalues $\la,\mu\in \cO_K$ of $A,B$, respectively, such that $T({\bf u})={\bf v}$, and $\la$, $\mu$ have the same prime ideal divisors in 
%$\cO_K$.
\end{prop}

\begin{proof}
%By Lemma \ref{l:suff}, it is enough to show the necessary part of the proposition.
%Furthermore, it is enough to show that
%if $T(G_A)=G_B$ for $T\in\GL_n(\bQ)$, then there exist eigenvectors ${\bf u},{\bf v}\in K ^n$ corresponding to
%eigenvalues $\la,\mu\in \cO_K$ of $A,B$, respectively, such that $T({\bf u})=x{\bf v}$ for $x\in K^{\times}$ and $\la$, $\mu$ have the same prime ideal divisors in $\cO_K$. Indeed, 
%the rest of the necessary conditions follow from Theorem \ref{th:prelim0}.
By Lemma \ref{l:divide}, $||S||=1$ and each $S_i$ is non-empty. Hence $k=\al_1=1$, $|S_1|=1$ and $h_A$ is irreducible. By symmetry, $h_B$ is irreducible and $T$ takes an eigenvector of $A$ to an eigenvector of $B$. Assume 
$A{\bf u}=\la{\bf u}$, $B{\bf v}=\mu{\bf v}$ for some $\la,\mu\in\overline{\bQ}$. Without loss of generality, we can assume 
${\bf u}\in\bQ(\la)^n$. From $T{\bf u}={\bf v}$ we have $BT{\bf u}=B{\bf v}=\mu T{\bf u}$. Since $B,T$ are defined over $\bQ$, this implies $\mu\in\bQ(\la)$ and hence $\bQ(\mu)=\bQ(\la)$. 

\sbr

We now show the existence of eigenvalues of $A$, $B$ sharing the same 
prime ideal divisors in the ring of integers $\cO_K$ of $K$. The argument is the same as in the proof of \cite[Proposition 4.1]{s}. We repeat it for the sake of completeness. By the previous paragraph, there exist 
$\mu\in \cO_K$ and an eigenvector ${\bf u}\in \cO_K ^n$ corresponding to an eigenvalue 
 $\la\in \cO_K$ of $A$ such that $T({\bf u})$ is an eigenvector of $B$ corresponding to $\mu$. Since $T(G_A)=G_B$, by definition \eqref{eq:opr} of groups $G_A$, $G_B$, for any $m\in\bN$ we have
\bbe\label{eq:llll}
B^{k_m}T=P_mA^{m},\quad k_m\in\bN\cup\{0\},\,\, P_m\in \operatorname{M}_n(\bZ).
\ee
Let $T=\frac{1}{l}T'$ for some $l\in\bZ-\{0\}$ and non-singular $T'\in \operatorname{M}_n(\bZ)$. Let $\fp$ be a prime ideal of $\cO_K$ % lying above $p$ 
that divides $\la$. By above, 
%$T{\bf u}$ is an eigenvector of $B$ corresponding to $\mu$, {\it i.e.},
$
B(T{\bf u})=\mu (T{\bf u}).
$
Hence, multiplying \eqref{eq:llll} by ${\bf u}$, we get
\bbe\label{eq:divi}
\mu^{k_m}T{\bf u}=B^{k_m}T{\bf u}=P_mA^{m}{\bf u}=P_m\la^m{\bf u},\quad\forall m\in\bN.
\ee
Here $T{\bf u}\ne{\bf 0}$, $T{\bf u}$ does not depend on $m$, and $\fp$ divides $\la$. This implies that $\fp$ divides $\mu$ ({\it e.g.}, this follows from the existence and uniqueness of decomposition of non-zero ideals into prime ideals in the Dedekind domain $\cO_K$). Analogously, it follows from \eqref{eq:divi} that all prime (ideal) divisors of $\la$ also divide $\mu$ (in $\cO_K$). Repeating the same argument with $A$ replaced by $B$ and $\la$ replaced by $\mu$, we see that all prime divisors of $\mu$ also divide $\la$. Thus, $\la$ and $\mu$ have the same prime divisors.
\end{proof}

\begin{example}
We demonstrate how Lemma \ref{l:divide} %and Proposition \ref{pr:main31} 
can be used to describe irreducible isomorphisms when $2\leq n\leq 4$. If $n=2,3$, then any irreducible isomorphism between $G_A,G_B$ 
implies $h_A,h_B$ are irreducible by Proposition \ref{pr:main31}. Let $n=4$ and assume there is an irreducible isomorphism between  $G_A,G_B$. Using properties (a)--(d) in Lemma \ref{l:divide} and Proposition \ref{pr:main31}, one can show that either $h_A$ is irreducible or $h_A=h_1h_2$, where $h_1,h_2\in\bZ[t]$ are irreducible of degree $2$ and, analogously, for $h_B$.
%$h_B$ is irreducible, or $h_A=h_1h_2$, $h_B=r_1r_2$, where $h_1,h_2,r_1,r_2\in\bZ[t]$ are irreducible of degree $2$. 
In particular, {\em e.g.}, one cannot have $h_A=f_1f_2$, where $f_1,f_2\in\bZ[t]$, $f_1$ is linear, and $f_2$ is irreducible of degree $3$.
\end{example}

\section{Irreducible characteristic polynomials, ideal classes}

We first show that in the case of irreducible characteristic polynomials $h_A$, $h_B$,
it is enough to assume that $T$ takes an eigenvector of $A$ 
to an eigenvector of $B$ for $T(G_A)=G_B$.

\begin{lem}\label{l:suff}
Let $A,B\in\M_n(\bZ)$ be non-singular and let $G_A,G_B$ have characteristics \eqref{eq:char10}, \eqref{eq:char20}, respectively. Assume the characteristic polynomials
of $A$, $B$ are irreducible. 
%$\chi_A\in\bZ[t]$ of 
%$A$ is irreducible.
%Let $K\subset \overline{\bQ}$ denote a finite Galois extension of $\bQ$ that contains the eigenvalues of both $A$ and $B$. 
Assume there exist eigenvalues $\la,\mu\in \cO_K$ corresponding to eigenvectors ${\bf u},{\bf v}\in K ^n$ of $A,B$, respectively, such that $\la$, $\mu$ have the same prime ideal divisors in the ring of integers of $K$. Then
$\cP=\cP(A)=\cP(B)$, $\cP'=\cP'(A)=\cP'(B)$, and $\caR=\caR(A)=\caR(B)$.
If $T\in\GL_n(\caR)$, $T({\bf u})={\bf v}$, 
%$($hence, \eqref{eq:class} holds$)$, 
and $T$ $($resp., $T^{-1})$ satisfies the condition $(A,B,p)$  
$($resp., $(B,A,p))$ for any $p\in\cP'$, then $T(G_A)=G_B$.
\end{lem}

\begin{proof}
By enlarging $K$ if necessary, without loss of generality, we can assume that $K$ is Galois over $\bQ$. For any
$\sg\in\Gal(K/\bQ)$, $\sg(\la)$ and $\sg(\mu)$ have the same prime ideal divisors. Thus, since $\Gal(K/\bQ)$ acts transitively on roots of irreducible polynomials $h_A,h_B\in\bZ[t]$, we have $t_p(A)=t_p(B)$, $\cP(A)=\cP(B)$, 
$\cP'=\cP'(A)=\cP'(B)$, and hence $\caR(A)=\caR(B)$. Furthermore, for $p\in\cP'$, a prime ideal $\fp$ of $\cO_K$ above $p$, and $\sg\in\Gal(K/\bQ)$, $\sg({\bf u})$ (resp., $\sg({\bf v})$) is an eigenvector of $A$ (resp., $B$) corresponding to $\sg(\la)$ (resp., $\sg(\mu)$) and $T(\sg({\bf u}))=\sg({\bf v})$, since $A,B,T$ are defined over $\bQ$. Thus,
$T(X_{A,\fp})=X_{B,\fp}$ and the lemma follows from Theorem \ref{th:prelim0}.
\end{proof}

\begin{rem}
%We know that Lemma \ref{l:suff} is also necessary if $n=2$ or $n=3$. 
We know that when $G_A\cong G_B$ and $n\geq 4$, not every isomorphism 
between $G_A$ and $G_B$ takes an eigenvector of $A$ to an eigenvector of $B$ (see Example \ref{ex:5} below). Also, in general, if $n>2$, $G_A\cong G_B$, and the characteristic polynomial of $A$ is irreducible, then not necessarily the characteristic polynomial of $B$ is also irreducible
(see Example \ref{ex:6} below).
\end{rem}

%\begin{rem} ???
%The case when the characteristic polynomial of $A$ is not irreducible or if $A$ is not diagonalizable over $\bQ$. Use Lemma appendix below, see Examples \ref{ex:rat1}, \ref{ex:rat2} ($n=2$ diagonalizable over $\bQ$) and also \cite{s} for  various cases when $n=2$, use Theorem \ref{th:main1} and ideas of this section in terms of eigenvectors. More examples $n=3$ ???
%\end{rem}

%\begin{rem}\label{r:char}
%Note that one can calculate a characteristic of $G_A$ over an extension of 
%$\bQ_p$ for each prime $p$. More precisely, let $K$ be a finite Galois extension of $\bQ$ that contains all the eigenvalues of $A$. As in the proof of Lemma \ref{lem:gen-n}, using an analogous construction, one can show that there exists a free basis $\{{\bf f}_1,\ldots,{\bf f}_n \}$ of $\bZ^n$ such that for any $p\in\cP'$ 
%and any prime $\fp$ of $\cO_K$ that divides $p$ we have
%$$
%\overline{G}_A\otimes_{\bZ_p}\cO_{\fp}=<({\bf f}_i+\sum_{j=t_p+1}^n\al_{pij}{\bf f}_j)\cdot\pi^{-\infty},{\bf f}_{t_p+1},\ldots,{\bf f}_n\,\vert\,1\leq i\leq t_p>,
%$$
%where $\pi$ is a uniformizer of $\cO_{\fp}$ and all $\al_{pij}\in\bZ_p$. This follows from $A$ being defined over $\bZ$ by applying the Galois actions on groups $
%\overline{G}_A\otimes_{\bZ_p}\cO_{\fp}$.
%It is helpful, since then we can use a Jordan canonical basis of $A$ (see Example \ref{ex:aaat} below). ???
%\end{rem}

We now recall generalized ideal classes introduced in \cite{s}.
Let $A,B\in\M_n(\bZ)$ be non-singular and let 
$\la\in\overline{\bQ}$ be an eigenvalue of $A$ corresponding to an eigenvector 
${\bf u}=\left(\begin{matrix}
u_1 & u_2 & \ldots & u_n
\end{matrix}
\right)^t\in \bQ(\la)^n$  of $A$. For the rest of this section we assume that the characteristic polynomials of $A$, $B$ are irreducible. Denote 
\begin{eqnarray*}
I_{\bZ}(A,\la)&=&\left\{ m_1 u_1+\cdots +m_n u_n\,\vert\,m_1,\ldots,m_n\in\bZ\right\}\subset \bQ(\la), \\
I_{\caR}(A,\la)&=&I_{\bZ}(A,\la)\otimes_{\bZ}\caR\subset \bQ(\la),\,\,\caR=\caR(A),
\end{eqnarray*}
where $\caR$ is given by \eqref{eq:rrr}.
Since $\la{\bf u}=A{\bf u}$ and $A$ has integer entries, $I_{\bZ}(A,\la)$ is a $\bZ[\la]$-module and 
$I_{\caR}(A,\la)$ is an $\caR[\la]$-module. Let 
$\mu\in\overline{\bQ}$ be an eigenvalue of $B$, and let $K$ be a number field 
with ring of integers $\cO_K$ such that $\la,\mu\in\cO_K$. Assume 
$\caR=\caR(A)=\caR(B)$ (which is a necessary condition for $G_A\cong G_B$). There exists 
$T\in\GL_n(\caR)$ such that $T(\bf u)$ is an eigenvector of $B$ corresponding to $\mu$ if and only if 
$$
I_{\caR}(A,\la)=yI_{\caR}(B,\mu),\quad y\in K^{\times},
$$
denoted by 
$[I_{\caR}(A,\la)]=[I_{\caR}(B,\mu)]$. We know that 
$[I_{\caR}(A,\la)]=[I_{\caR}(B,\mu)]$ is among sufficient conditions for $G_A\cong G_B$ for any $n\geq 2$ (Lemma \ref{l:suff} above). In \cite[Theorem 6.6]{s} we prove that this is also a necessary condition when $n=2$. 
Proposition \ref{pr:main3} below extends the result to an arbitrary $n$ under an additional assumption that there exists $t_p$ coprime with $n$ (denoted by $(n,t_p)=1$). In fact, the proposition shows more, namely, than {\em any}  isomorphism takes an eigenvector of $A$ to an eigenvector of $B$. It turns out that $[I_{\caR}(A,\la)]=[I_{\caR}(B,\mu)]$ is not a necessary condition for $G_A\cong G_B$ for an arbitrary $n$ (see Example \ref{ex:5} below, where the condition 
$(n,t_p)=1$ does not hold). The next proposition is a direct consequence of Proposition \ref{pr:main31}, since if the characteristic polynomial of $A$ is irreducible, then clearly, any isomorphism between $G_A$, $G_B$ is irreducible.

%We can show that this is also a necessary condition when $n=2,3$. We do not know whether it holds for any $n\geq 2$. More precisely, if $n=2,3$, the characteristic polynomials of $A$, $B$ are irreducible, and $G_A\cong G_B$, then
%$\caR=\caR(A)=\caR(B)$ and
%$I_{\caR}(A,\la)=yI_{\caR}(B,\mu)$ for some $y\in K^{\times}$.
%, where $K$ is the splitting field of $h_A$, which is also the splitting field of $h_B$.

\begin{prop}\label{pr:main3}
Let $A,B\in\M_n(\bZ)$ be non-singular. Assume the characteristic polynomial of $A$ is irreducible  and there exists a prime $p\in\cP'(A)$ with $(n,t_p(A))=1$. Let $K\subset \overline{\bQ}$ be a finite extension of $\bQ$ that contains the eigenvalues of both $A$ and $B$.
If $T\in\GL_n(\bQ)$ is an isomorphism from $G_A$ to $G_B$ $($equivalently, $T(G_A)=G_B)$, then there exist eigenvectors ${\bf u},{\bf v}\in K ^n$
corresponding to eigenvalues $\la,\mu\in \cO_K$ of $A,B$, respectively, such that $T({\bf u})={\bf v}$, and $\la$, $\mu$ have the same prime ideal divisors in 
$\cO_K$.
\end{prop}

Combining Proposition \ref{pr:main3} with Lemma \ref{l:suff} and Theorem \ref{th:prelim0}, we get the following necessary and sufficient criterion for $G_A\cong G_B$ under the additional condition in Proposition \ref{pr:main3}.

\begin{prop}\label{pr:main3iff}
Let $A,B\in\M_n(\bZ)$ be non-singular with irreducible characteristic 
polynomials and let $G_A,G_B$ have characteristics \eqref{eq:char10}, \eqref{eq:char20}, respectively. Assume there exists a prime $p$ with 
$(t_p(A),n)=1$. 
Let $K\subset \overline{\bQ}$ be a finite extension of $\bQ$ that contains the eigenvalues of both $A$ and $B$.
%Assume the characteristic polynomial of $A$ is irreducible. 
Then $T\in\GL_n(\bQ)$ is an isomorphism from $G_A$ to $G_B$ if and only if there exist eigenvalues $\la,\mu\in \cO_K$ corresponding to eigenvectors ${\bf u},{\bf v}\in K ^n$ of $A,B$, respectively, such that $\la$, $\mu$ have the same prime ideal divisors in $\cO_K$, 
%\begin{eqnarray*}
%\caR&=&\caR(A)=\caR(B),  \\
%\cP&=&\cP(A)=\cP(B), \\
%\cP'&=&\cP'(A)=\cP'(B),
%\end{eqnarray*}
$T\in\GL_n(\caR)$, $T({\bf u})={\bf v}$, and $T$ $($resp., $T^{-1})$ satisfies the condition $(A,B,p)$  
$($resp., $(B,A,p))$ for any $p\in\cP'$.
\end{prop}

In the case $n=2$, to decide whether $G_A$ and $G_B$ are isomorphic, we can omit conditions $(A,B,p)$, $(B,A,p)$.

\begin{prop}\cite[Theorem 6.6]{s}\label{pr:main2}
Let $A,B\in\M_2(\bZ)$ be non-singular. Assume the characteristic polynomial of $A$ is irreducible and $\cP'(A)\ne\emptyset$. 
Then $G_A \cong G_B$ if and only if there exist eigenvalues $\la,\mu\in \cO_K$ %corresponding to eigenvectors ${\bf u},{\bf v}\in K ^2$ 
of $A,B$, respectively, such that $\la$, $\mu$ have the same prime ideal divisors in $\cO_K$
and
$$
[I_{\caR}(A,\la)]=[I_{\caR}(B,\mu)],\quad \caR=\caR(A).
$$
\end{prop}

Proposition \ref{pr:main2} can be generalized to an arbitrary $n$ under an additional condition, which automatically holds when $n=2$. Namely, $t_p=n-1$ for any $p\in\cP'$.

\begin{lem}\label{l:mainn}
Let $A,B\in\M_n(\bZ)$ be non-singular with irreducible characteristic polynomials, $\cP'(A)\ne\emptyset$, and $t_p(A)=n-1$ for any $p\in\cP'(A)$. 
Then $G_A \cong G_B$ if and only if there exist eigenvalues $\la,\mu\in \cO_K$ %corresponding to eigenvectors ${\bf u},{\bf v}\in K ^n$ 
of $A,B$, respectively, such that $\la$, $\mu$ have the same prime ideal divisors in $\cO_K$
and
$$
[I_{\caR}(A,\la)]=[I_{\caR}(B,\mu)].
$$
\end{lem}
\begin{proof}
By Proposition \ref{pr:main3iff}, it is enough to show the sufficient part. As in the proof of Lemma \ref{l:suff}, 
we have
$$\cP=\cP(A)=\cP(B),\,\, \cP'=\cP'(A)=\cP'(B),\,\,\caR=\caR(A)=\caR(B),
$$
and $t_p=t_p(A)=t_p(B)$ for any prime $p\in\bN$. Note that $[I_{\caR}(A,\la)]=[I_{\caR}(B,\mu)]$ is equivalent to the existence of 
$T\in\GL_n(\caR)$ such that $T(\bf u)$ is an eigenvector of $B$ corresponding to $\mu$ for an eigenvector ${\bf u}$ of $A$ corresponding to $\la$. As in the 
proofs of Theorem \ref{th:prelim0} and Lemma \ref{l:suff}, such $T$ induces an isomorphism between the divisible parts $D_p(A)$ and $D_p(B)$ of $\overline{G}_{A,p}$ and
$\overline{G}_{B,p}$, respectively, for any $p$. 
Under the assumption $t_p=n-1$, $p\in\cP'$, the reduced parts $R_p(A)$ and $R_p(B)$ of $\overline{G}_{A,p}$ and
$\overline{G}_{B,p}$, respectively,  are free $\bZ_p$-modules of rank $1$. Hence, 
there exists $k\in\bZ$ such that for
$T'=p^kT$ we have 
\bbe\label{eq:tt'}
T'(R_p(A))\subseteq D_p(B)\oplus R_p(B)\quad\text{ and }\quad
(T')^{-1}(R_p(B))\subseteq D_p(A)\oplus R_p(A).
\ee
Indeed, as follows from \eqref{eq:dec1} and \eqref{eq:dec2}, 
$T({\bf e}_n)=a+y{\bf e}_n$ for some $a\in D_p(B)$ and $y\in\bQ_p$. Let $y=p^{-k}u$ for some $k\in\bZ$ and $u\in\bZ_p^{\times}$. Then 
$T'({\bf e}_n)=p^ka+u{\bf e}_n$, where $p^ka\in D_p(B)$ and hence $T'({\bf e}_n)\in\overline{G}_{B,p}$, since $\overline{G}_{B,p}=D_p(B)\oplus\bZ_p{\bf e}_n$. Clearly, $T'$ still induces an isomorphism between $D_p(A)$, $D_p(B)$, and $T'\in\GL_n(\caR)$, since $p\in\cP'$. Moreover, for a prime $q$ distinct from $p$, 
$qT'$ also satisfies \eqref{eq:tt'}, since $q\in\bZ_p^{\times}$. 
Since $\cP'$ is finite, it shows that there exists $a\in\caR^{\times}$ such that $aT\in\GL_n(\caR)$ is an isomorphism from $\overline{G}_{A,p}$ to $\overline{G}_{B,p}$ for any $p\in\cP'$ and hence $aT$ is an isomorphism from ${G}_{A}$ to ${G}_{B}$ by
Corollary \ref{cor:loc'}.
\end{proof}

\section{Examples}

\begin{example} One of the easiest examples is when $\cP'=\emptyset$. Let 
$$
A=\left(\begin{matrix}
0 & 4 \\
2 & 0 
\end{matrix}
\right),\,\,
B=\left(\begin{matrix}
0 & 8 \\
1 & 0 
\end{matrix}
\right).
$$
Both $A$ and $B$ have the same characteristic polynomial $x^2-8$, irreducible over $\bQ$, so that $A$ and $B$ are conjugate over $\bQ$ and have the same eigenvalues. There is only one prime $p=2$ that divides $\det A$ and it also divides $\tr A=0$. Hence, by Lemma \ref{l:repr'},
$$
G_A=G_B=<{\bf e}_1,{\bf e}_2,2^{-\infty}{\bf e}_1, 2^{-\infty}{\bf e}_2>.
$$
In general, if 
$h_A\equiv x^n\,(\text{mod }p)$ for any prime $p$ that divides $\det A$, then
$$
G_A=<p^{-k}{\bf e}_i\,\vert\,\,i\in\{1,2,\ldots,n\},\,\,p\,\vert\det A,\,k\in\bN\cup\{0\}>.
$$
\end{example}

\begin{example}\label{ex:rat1}
In this and the next examples we show how Theorem \ref{th:main1} can be effectively used in the case when the characteristic polynomials are not irreducible.
Let 
$$
A=\left(\begin{matrix}
88 & -68 \\
34 & -14 
\end{matrix}
\right),\,\,
B=\left(\begin{matrix}
-192 & 304 \\
-144 & 248 
\end{matrix}
\right).
$$
Here $A$ has eigenvalues $20,54$ and $B$ has eigenvalues $-40,96$. Let
\begin{align*}
\la_1&=20=2^2\cdot 5,\\
\la_2&=54=2\cdot 3^3,\\
\mu_1&=-40=-2^3\cdot 5,\\
\mu_2&=96=2^5\cdot 3.
%\la_2-\la_1&=34=2\cdot 17,\\
%\mu_2-\mu_1&=136=2^3\cdot 17.
\end{align*}
Thus, 
\begin{eqnarray*}
\cP&=&\cP(A)=\cP(B)=\{2,3,5\}, \\
\cP'&=&\cP'(A)=\cP'(B)=\{3,5\},\\
t_3&=&t_3(A)=t_3(B)=1,\\
t_5&=&t_5(A)=t_5(B)=1,\\
\caR&=&\caR(A)=\caR(B)=\left\{n2^k3^l5^m\,\,\vert\,\,k,l,m,n\in\bZ\right\}.
\end{eqnarray*}
We have
$$
A=S\left(\begin{matrix}
\la_1 & 0 \\
0 & \la_2 
\end{matrix}
\right)S^{-1},\quad S=\left(\begin{matrix}
1 & 2 \\
1 & 1 
\end{matrix}
\right)=
\left(\begin{matrix}
{\bf u}_1 & {\bf u}_2
\end{matrix}
\right)
\in\operatorname{GL}_2(\bZ).
$$
Thus, in the notation of Lemma \ref{l:repr'}, 
$W_5=W_5(A)=S$, $W_3=W_3(A)=\left(\begin{matrix}
{\bf u}_2 & {\bf u}_1
\end{matrix}
\right)$, and
$$
G_A=<{\bf u}_1,{\bf u}_2,2^{-\infty}{\bf u}_1, 2^{-\infty}{\bf u}_2,
5^{-\infty}{\bf u}_1, 3^{-\infty}{\bf u}_2>.
$$
%Starting with vectors ${\bf u}_1$, ${\bf u}_2$ and following the proof of Lemma \ref{lem:gen-n} above, one can find a free basis $\{{\bf f}_1,{\bf f}_2\}$ and the characteristic of $G_A$ with respect to the basis. Namely,
%for ${\bf f}_1={\bf u}_1$, ${\bf f}_2={\bf u}_2-5{\bf u}_1$ we have
%$$
%G_A=<{\bf f}_1,{\bf f}_2,2^{-\infty}{\bf f}_1, 2^{-\infty}{\bf f}_2,
%5^{-\infty}{\bf f}_1, 3^{-\infty}({\bf f}_1+5^{-1}{\bf f}_2)>,
%$$
%since $5\in\bZ_3^{\times}$.
%Thus,
%$$
%M(A;{\bf f}_1, {\bf f}_2)=\{\al_{512}(A)=0, \al_{312}(A)=5^{-1}\}
%$$ is the characteristic of $G_A$ with respect to $\{{\bf f}_1, {\bf f}_2\}$. 
Also,
$$
G_A=<{\bf e}_1,{\bf e}_2,2^{-\infty}{\bf e}_1, 2^{-\infty}{\bf e}_2,
5^{-\infty}({\bf e}_1+{\bf e}_2), 3^{-\infty}({\bf e}_1+2^{-1}{\bf e}_2)>,
$$
since $2\in\bZ_3^{\times}$.
Thus,
$$
M(A;{\bf e}_1, {\bf e}_2)=\{\al_{512}(A)=1, \al_{312}(A)=2^{-1}\}
$$ is the characteristic of $G_A$ with respect to $\{{\bf e}_1, {\bf e}_2\}$.
Similarly, we find a characteristic of $G_B$. 
One can show that
$$
B=P\left(\begin{matrix}
\mu_1 & 0 \\
0 & \mu_2 
\end{matrix}
\right)P^{-1},\quad P=\left(\begin{matrix}
2 & 19 \\
1 & 18 
\end{matrix}
\right)=\left(\begin{matrix}
{\bf v}_1 & {\bf v}_2
\end{matrix}
\right)\in\operatorname{M}_2(\bZ).
$$
Note that $\det P=17\in\bZ_p^{\times}$ for any $p\in\cP'=\{3,5\}$. Thus, in the notation of Lemma \ref{l:repr'}, 
$W_5=W_5(B)=P$, $W_3=W_3(B)=\left(\begin{matrix}
{\bf v}_2 & {\bf v}_1
\end{matrix}
\right)$, and
\begin{eqnarray*}
G_B&=&<{\bf e}_1,{\bf e}_2,2^{-\infty}{\bf e}_1, 2^{-\infty}{\bf e}_2,
5^{-\infty}{\bf v}_1, 3^{-\infty}{\bf v}_2>= \\
&=&<{\bf e}_1,{\bf e}_2,2^{-\infty}{\bf e}_1, 2^{-\infty}{\bf e}_2,
5^{-\infty}({\bf e}_1+2^{-1}{\bf e}_2), 3^{-\infty}({\bf e}_1+\frac{18}{19}{\bf e}_2)>,
\end{eqnarray*}
since $2\in\bZ_5^{\times}$, $19\in\bZ_3^{\times}$.
Thus,
$$
M(B;{\bf e}_1, {\bf e}_2)=\left\{\al_{512}(B)=2^{-1}, \al_{312}(B)=\frac{18}{19}\right\}
$$ 
is the characteristic of $G_B$ with respect to 
$\{{\bf e}_1,{\bf e}_2\}$. Using Theorem \ref{th:main1}, one can show that $G_A$ is not isomorphic to $G_B$. Namely, one can show that if $T\in\GL_2(\bQ)$ and
$T({\bf u}_i)=m_i{\bf v}_i$, $i=1,2$, for some $m_1,m_2\in\bQ$, then $T\not\in\GL_2(\caR)$. 
\end{example}

\begin{example}\label{ex:rat2}
%Not irreducible!
%We keep the notation of the previous example, Example \ref{ex:rat}.
Let 
$$
C=\left(\begin{matrix}
87 & -67 \\
33 & -13
\end{matrix}
\right),\,\,
B=\left(\begin{matrix}
-192 & 304 \\
-144 & 248 
\end{matrix}
\right),
$$
where $C$ has eigenvalues $\la_1=20,\la_2=54$, and $B$ is the same as in Example \ref{ex:rat1}. We claim that $G_C\cong G_B$.
Indeed, 
$$
C=S\left(\begin{matrix}
\la_1 & 0 \\
0 & \la_2 
\end{matrix}
\right)S^{-1},\quad S=\left(\begin{matrix}
1 & -67 \\
1 & -33 
\end{matrix}
\right)=\left(\begin{matrix}
{\bf w}_1 & {\bf w}_2
\end{matrix}
\right)\in\operatorname{M}_2(\bZ).
$$
We have $\cP=\{2,3,5\}$, $\cP'=\{3,5\}$, $t_3=1$, $t_5=1$. Since $\det S=34\in\bZ_p^{\times}$ for any $p\in\cP'$, by Lemma \ref{l:repr'}, 
$W_5=W_5(C)=S$, $W_3=W_3(C)=\left(\begin{matrix}
{\bf w}_2 & {\bf w}_1
\end{matrix}
\right)$, and
\begin{eqnarray*}
G_C&=&<{\bf e}_1,{\bf e}_2,2^{-\infty}{\bf e}_1, 2^{-\infty}{\bf e}_2,
5^{-\infty}{\bf w}_1, 3^{-\infty}{\bf w}_2>=\\
&=&<{\bf e}_1,{\bf e}_2,2^{-\infty}{\bf e}_1, 2^{-\infty}{\bf e}_2,
5^{-\infty}({\bf e}_1+{\bf e}_2), 3^{-\infty}({\bf e}_1+\frac{33}{67}{\bf e}_2)>,
\end{eqnarray*}
since $67\in\bZ_3^{\times}$.
Thus, 
$$
M(C;{\bf e}_1, {\bf e}_2)=\left\{\al_{512}(C)=1, \al_{312}(C)=\frac{33}{67}\right\}
$$ 
is the characteristic of $G_C$ with respect to 
$\{{\bf e}_1, {\bf e}_2\}$. Using Theorem \ref{th:main1}, one can find $T\in\GL_2(\caR)$ such that 
$T({\bf v}_i)=m_i{\bf w}_i$, $m_i\in\bQ$, $i=1,2$. For example,
$$
T=\left(\begin{matrix}
5 & -9 \\
3 & -5
\end{matrix}
\right),\quad\det T=2\in\caR^{\times},
$$
the conditions in Theorem \ref{th:main1} are satisfied and, hence, $T:G_B\rar G_C$ is an isomorphism.
\end{example}

There are several examples in \cite{s} when $n=2$ and characteristic polynomials are irreducible. We now look at higher-dimensional examples.

\begin{example}\label{ex:aaa} 
In this and the next examples we show two ways to compute characteristics.
 Let $n=3$, $h=t^3+t^2+2t+6$, and 
$$
A=\left(\begin{matrix}
0 & 0 & -6 \\
1 & 0 & -2 \\
0 & 1 & -1 
\end{matrix}
\right),
$$
a rational canonical form of $h$.
%, so that $h$ is the characteristic polynomial of $A$. 
Note that $h\in\bZ[t]$ is irreducible in $\bQ[t]$. We will compute 
the characteristic of $G_A$ with respect to the standard basis $\{{\bf e}_1,{\bf e}_2,{\bf e}_3\}$. The calculation is justified by the proof of Theorem \ref{th:trig}.

\sbr

We have $\det A=-6$,  
$\cP=\cP'=\{2,3\}$. 
Let $p=2$. Then
$$
h\equiv t^2\cdot (t+1) \,(\text{mod }2),\quad \overline{G}_{A,p}\cong \bQ^2_p\oplus\bZ_p,\quad t_p=2,
$$
by Proposition \ref{prop:1} above. As follows from the proof of Lemma \ref{lem:gen-n},
to determine a characteristic of $G_A$, we need to find generators of the divisible part $D_p(A)$ of $\overline{G}_{A,p}$, {\it i.e.}, a
$\bZ_p$-submodule of $\overline{G}_{A,p}$ isomorphic to $\bQ_p^2$. By Hensel's lemma,  
$h=(t-\la)g(t)$, where $\la\in\bZ_p^{\times}$ and $g\in\bZ_p[t]$ is of degree $2$. One can show that $g$ is irreducible over $\bQ_p$. Let $\al\in\overline{\bQ}_p$ be a root of $g$. Let ${\bf u}(\al)\in\bZ_p[\al]^3$ denote an eigenvector
of $A$ corresponding to $\al$. We can take
$$
{\bf u}(\al)=\left(\begin{matrix}
 -6 \\
\al(\al+1) \\
\al 
\end{matrix}
\right)=C\left(\begin{matrix}
 1 \\
\al
\end{matrix}
\right),\quad
C=\left(\begin{matrix}
 -6 & 0 \\
6\la^{-1} & -\la \\
0 & 1 
\end{matrix}
\right)\in\M_{3\times 2}(\bZ_p).
$$
We then look for a Smith normal form of $C$:
$$
C=U\left(\begin{matrix}
 -6 & 0 \\
0 & -\la \\
0 & 0 
\end{matrix}
\right),\quad
U=\left(\begin{matrix}
 1 & 0 & 0 \\
-\la^{-1} &1 & 0 \\
0 & -\la^{-1} & 1
\end{matrix}
\right)\in\GL_3(\bZ_p).
$$
The first two columns ${\bf u}_{21}$, ${\bf u}_{22}$ of $U$ give us generators of 
$D_p(A)$:
$$
{\bf u}_{21}=\left(\begin{matrix}
 1  \\
-\la^{-1} \\
0 \end{matrix}
\right),\quad
{\bf u}_{22}=\left(\begin{matrix}
  0 \\
1 \\
-\la^{-1}
\end{matrix}
\right).
$$
Analogously, for $p=3$ we have
$$
h\equiv t\cdot (t^2+t+2) \,(\text{mod }3),\quad
\overline{G}_{A,p}\cong \bQ_p\oplus\bZ_p^2,\quad t_p=1,
$$
by Proposition \ref{prop:1} above. By Hensel's lemma, $h$ has a root $\ga\in p\bZ_p$.  As a generator of $D_p(A)$, we can take an eigenvector 
${\bf u}_{31}={\bf u}(\ga)$ of $A$ corresponding to  
$\ga$. 
%$$
%{\bf u}_{31}=\left(\begin{matrix}
% -6 \\
%\ga(\ga+1) \\
%\ga 
%\end{matrix}
%\right). 
%$$
By Lemma \ref{l:repr'},
$$
G_A=<{\bf e}_1,{\bf e}_2,2^{-\infty}{\bf u}_{21},2^{-\infty}{\bf u}_{22}, 
3^{-\infty}{\bf u}_{31}>.
$$
We now change the system $\{{\bf u}_{ij}\}$ so that it has the form \eqref{eq:form} with respect to $\{{\bf e}_1, {\bf e}_2,{\bf e}_3\}$. 
%In particular, 
%$$
%\la=1+2\cdot +\cdots   \in\bZ_2,
%$$
%$$
%\ga=2\cdot 3+3^3+2\cdot 3^4+2\cdot 3^8+\cdots   \in\bZ_3
%$$
%can be found from the Newton ??? method as a roots of $h$ considered as an element of $\bZ_2[t]$ and $\bZ_3[t]$, respectively. 
%We now find a characteristic of $G_A$. We have
%\begin{eqnarray*}
%{\bf u}_{21}&=&{\bf e}_1-\la^{-1}{\bf e}_2,\\
%{\bf u}_{22}&=&{\bf e}_2-\la^{-1}{\bf e}_3, \\
%{\bf u}_{31}&=&-6{\bf e}_1+\ga(\ga+1){\bf e}_2+\ga{\bf e}_3.
%\end{eqnarray*}
For ${\bf x}_{21}={\bf u}_{21}+\la^{-1}{\bf u}_{22}$, ${\bf x}_{22}={\bf u}_{22}$, ${\bf x}_{31}=(-1/6){\bf u}_{31}$, we have
\begin{eqnarray*}
{\bf x}_{21}&=&{\bf e}_1-\la^{-2}{\bf e}_3,\quad p=2,\\
{\bf x}_{22}&=&{\bf e}_2-\la^{-1}{\bf e}_3,\quad p=2, \\
{\bf x}_{31}&=&{\bf e}_1-(1/2)(\ga/3)(\ga+1){\bf e}_2-(1/2)(\ga/3){\bf e}_3, \quad p=3.
\end{eqnarray*}
Note that in ${\bf x}_{31}$, $2$ is a unit in $\bZ_3$ and $3$ divides $\ga$ in $\bZ_3$, so that $1/2,\ga/3\in\bZ_3$. Therefore,
$$
M(A;{\bf e}_1, {\bf e}_2,{\bf e}_3)=\{\al_{213},\al_{223}, \al_{312},\al_{313}\},
$$ 
where
\begin{alignat}{2}
\al_{213} & = -\la^{-2}, & \quad  \al_{312}&=-(1/2)(\ga/3)(\ga+1), \nonumber \\
\al_{223}& = -\la^{-1}, & \quad  \al_{313}&=-(1/2)(\ga/3)\nonumber 
\end{alignat}
is the characteristic of $G_A$ with respect to 
$\{{\bf e}_1, {\bf e}_2,{\bf e}_3\}$. 
\end{example}

\begin{example}\label{ex:aaat}
In this example we show another way to calculate a characteristic. We use Remark \ref{r:char} above that a characteristic can be calculated over an extension 
of $\bQ_p$ for each prime $p$. We find a characteristic of $G_{A^t}$, where $A$ is from Example \ref{ex:aaa} and $A^t$ is the transpose of $A$. Note that if $\de$ is an eigenvalue of $A$, then
${\bf v}(\de)=\left(\begin{matrix}
 1 & \de & \de^2 
\end{matrix}
\right)^t$
is an eigenvector of $A^t$ corresponding to $\de$. We use the notation of  Example \ref{ex:aaa}. For $p=2$, let 
$\al_1,\al_2\in\overline{\bQ}_p$ be (distinct) roots of $g$. By Lemma \ref{l:div}, 
${\bf v}(\al_1),{\bf v}(\al_2)$ are generators of the divisible part of $\overline{G}_{A,p}$ over the ring of integers of a finite extension of $\bQ_p$ that contains $\al_1,\al_2$. We now change $\{{\bf v}(\al_1),{\bf v}(\al_2)\}$ so that it has the form \eqref{eq:form}. Namely, let
\begin{eqnarray*}
{\bf v}_{22}&=&\frac{1}{\al_2-\al_1}({\bf v}(\al_2)-{\bf v}(\al_1))=
\left(\begin{matrix}
 0 & 1 & \al_1+\al_2 
\end{matrix}
\right)^t, \\
{\bf v}_{21}&=&{\bf v}(\al_1)-\al_1{\bf v}_{22}=
\left(\begin{matrix}
 1 & 0 & -\al_1\al_2 
\end{matrix}
\right)^t.
\end{eqnarray*}
Since $\al_1,\al_2,\la$ are roots of $h$ and $h=t^3+t^2+2t+6$, we have
 $\al_1+\al_2+\la=-1$ and $\al_1\al_2\la=-6$. Recall $\la\in\bZ_p^{\times}$. 
Hence,
\begin{eqnarray*}
{\bf v}_{21}&=&{\bf e}_1+6\la^{-1}{\bf e}_3,\\
{\bf v}_{22}&=&{\bf e}_2-(\la+1){\bf e}_3, \\
{\bf v}_{31}&=&{\bf v}(\ga)={\bf e}_1+\ga{\bf e}_2+\ga^2{\bf e}_3.
\end{eqnarray*}
Therefore,
$
M(A^t;{\bf e}_1, {\bf e}_2,{\bf e}_3)=\{\al'_{213},\al'_{223}, \al'_{312},\al'_{313}\},
$ 
where
\begin{alignat}{2}
\al'_{213} & = 6\la^{-1}, & \quad  \al'_{312}&=\ga, \nonumber \\
\al'_{223}& = -(\la+1), & \quad  \al'_{313}&=\ga^2\nonumber 
\end{alignat}
is the characteristic of $G_{A^t}$ with respect to the basis 
$\{{\bf e}_1, {\bf e}_2,{\bf e}_3\}$. 
\end{example}

\begin{example} 
Using Examples \ref{ex:aaa} and \ref{ex:aaat}, we show 
$G_A\cong G_{A^t}$.
Let $$
A=\left(\begin{matrix}
0 & 0 & -6 \\
1 & 0 & -2 \\
0 & 1 & -1 
\end{matrix}
\right),\quad
{\bf u}=\left(\begin{matrix}
 -6 \\
\de(\de+1) \\
\de 
\end{matrix}
\right),\quad 
{\bf v}=\left(\begin{matrix}
 1 \\
\de \\
\de^2 
\end{matrix}
\right),
$$
where ${\bf u}$, ${\bf v}$
are eigenvectors of $A$, $A^t$, respectively, corresponding to an eigenvalue $\de$. 
Thus,
\begin{eqnarray*}
\caR&=&\caR(A)=\caR(A^t)=\left\{n2^k3^l\,\,\vert\,\,n,k,l\in\bZ\right\},\\
I_{\caR}(A,\de)&=&\Span_{\caR}(-6,\de,\de(\de+1))=
\Span_{\caR}(1,\de,\de^2),
\end{eqnarray*}
since $6\in\caR^{\times}$, and 
$$
I_{\caR}(A^t,\de)=
\Span_{\caR}(1,\de,\de^2)=I_{\caR}(A,\de).
$$
We obtain $T\in\GL_3(\caR)$ by expressing coordinates of 
${\bf u}$ in terms of coordinates of ${\bf v}$:
%Therefore, $T=NM^{-1}$, where $M$ (resp., $N$) are matrices consisting of eigenvectors of $A$ (resp., $A^t$) defined by ${\bf u}$ (resp., ${\bf v}$) belongs to $\GL_3(\caR)$. In other words, we obtain $T$ by expressing coordinates of 
%${\bf v}$ in terms of coordinates of ${\bf u}$.
$$
T=\left(\begin{matrix}
-6 & 0 & 0 \\
0 &1& 1 \\
0 & 1 & 0
\end{matrix}
\right),\quad
T^{-1}=\left(\begin{matrix}
-1/6 & 0 & 0 \\
0 &0 & 1 \\
0 & 1 & -1
\end{matrix}
\right).
$$
Note that we were able to compute the characteristics of  both $G_A$, $G_{A^t}$ with respect to the standard basis, without having to change the basis (or, equivalently, conjugate $A$, $A^t$ by matrices in $\GL_3(\bZ)$). 
Therefore, $T$ $($resp., $T^{-1})$ satisfies the condition $(A,B,p)$  
$($resp., $(B,A,p))$ for any $p\in\cP'$, since $2$nd and $3$rd columns of both $T$, $T^{-1}$ consist of integers.
Since $T\in\GL_3(\caR)$, characteristics of both $A$, $A^t$ are with respect to the standard basis, and $A$, $A^t$ share the same eigenvalues, by Proposition \ref{pr:main3iff}, $T:G_{A^t}\rar G_{A}$ is an isomorphism. 
\end{example}

\begin{example}\label{ex:5}
Assume $A,B\in\M_n(\bZ)$ have irreducible characteristic polynomials.
By Proposition \ref{pr:main3iff}, if
$G_A\cong G_B$, then $[I_{\caR}(A,\la)]=[I_{\caR}(B,\mu)]$  under some additional conditions on $A$. In this example we show that this is not true in general. More precisely, 
$A,B\in\M_4(\bZ)$ share the same irreducible characteristic polynomial, $G_A\cong G_B$, but
$[I_{\caR}(A,\la)]\ne [I_{\caR}(B,\mu)]$. In particular, it shows that even when the characteristic polynomials of $A$, $B$ are irreducible and $G_A\cong G_B$, not every isomorphism between $G_A$ and $G_B$ takes an eigenvector of $A$ to an eigenvector of $B$ (unlike {\em e.g.}, the case of a prime dimension $n$). 
Here $n=4$ and $t_p=2$, so that the condition $(t_p,n)=1$ in Proposition \ref{pr:main3iff} does not hold. 

\sbr

Let $h(t)=t^4-2t^3+21t^2-20t+5$, irreducible over $\bQ$, and let $\la\in\overline{\bQ}$ be a root of $h$. 
%We use information about number field $K=\bQ(\la)$ from \cite{db} and \cite{sage}. 
By \cite{db}, $\cO_K=\bZ[\la]$, $K$ is Galois over $\bQ$, $\Gal(K/\bQ)\cong(\bZ/2\bZ)^2$, and the ideal class group of $K$ is non-trivial. % isomorphic to $(\bZ/2\bZ)^2$. 
Thus, there exists an ideal $J_1$ of $\bZ[\la]$ such that its ideal class $[J_1]$ is not trivial, {\it i.e.}, 
there is no $x\in K$ such that $J_1=x\bZ[\la]$. By \cite{sage}, we can take $J_1$ to be the ideal of $\bZ[\la]$ generated by $7$ and $\la^3-\la^2+20\la-4$ over $\bZ[\la]$, denoted by $J_1=(7,\la^3-\la^2+20\la-4)$. One can also find a $\bZ$-basis of $J_1$, {\it e.g.},
$J_1=\bZ[\om_1,\om_2,\om_3,\om_4]$, where
\begin{eqnarray*}
\om_1&=&7,\\
\om_2&=&2\la^3-3\la^2+41\la-16, \\
\om_3&=&\la^3-\la^2+20\la-4, \\
\om_4&=&-2\la^3+3\la^2-40\la+25.
\end{eqnarray*}
Since $[J_1]$ is non-trivial,
by Latimer--MacDuffee--Taussky Theorem \cite{t}, matrices $A$, $B$ corresponding to $(1)=\bZ[\la]$
and $J_1$, respectively, are not conjugated by a matrix from $\GL_4(\bZ)$. We find 
$A$, $B$ from the condition that
%$\la \la^{i-1}=\sum_{j=1}^4a_{ij}\la^{i-1}$, 
%$\la\om_i=\sum_{j=1}^4b_{ij}\om_j$, $i=1,\ldots,4$, respectively. Thus,
$$
{\bf u}=\left(\begin{matrix}
1 & \la & \la^2 & \la^3
\end{matrix}
\right)^t,\quad
{\bf v}=\left(\begin{matrix}
\om_1 & \om_2 & \om_3 & \om_4
\end{matrix}
\right)^t
$$
are eigenvectors of $A$, $B$, respectively, corresponding to $\la$. Thus,  
$$
A=\left(\begin{matrix}
0 & 1 & 0 & 0 \\
0 & 0 & 1 & 0 \\
0 & 0 & 0 & 1 \\
-5 & 20 & -21 & 2
\end{matrix}
\right),\quad
B=\left(\begin{matrix}
-9 & 7 & 0 & 7 \\
-6 & 4 & 1 & 4 \\
5 & -4 & 1 & -4 \\
-8 & 5 & 1 & 6
\end{matrix}
\right).
$$
Both $A$, $B$ have characteristic polynomial $h(t)=t^4-2t^3+21t^2-20t+5$, $\det A=\det B=5$, $\cP=\cP'=\{5\}$, $t_5=2$, and $[I_{\bZ}(A,\la)]\ne [I_{\bZ}(B,\la)]$. 
We show $[I_{\caR}(A,\la)]\ne[I_{\caR}(B,\la)]$, where $\caR=\{\frac{m}{5^k}\,\vert\,m,k\in\bZ \}$. Equivalently, we show that there is no $x\in K$ such that 
\bbe\label{eq:contr}
x(I_{\bZ}(A,\la)\otimes_{\bZ}\caR)=I_{\bZ}(B,\la)\otimes_{\bZ}\caR, 
\ee
where
\begin{eqnarray*}
I_{\bZ}(A,\la)&=&\bZ[1,\la,\la^2,\la^3]=(1), \\
I_{\bZ}(B,\la)&=&\bZ[\om_1,\om_2,\om_2,\om_3]=J_1.
\end{eqnarray*}
We also demonstrate how the standard methods of working with fractional ideals of $\cO_K$ (such as the prime ideal factorization and divisibility properties) can be used in the case of the ring $\caR$. This suggests the practicality of using generalized ideal classes. Assume there exists 
$x\in K$ satisfying \eqref{eq:contr}. Then $5^kx\in J_1$ for some $k\in\bN\cup\{0\}$. In particular, $y=5^kx\in\bZ[\la]$. Then $y\in J_1$ implies that 
$J_1$ divides the ideal $(y)=y\bZ[\la]$ of $\bZ[\la]$ generated by $y$, {\it i.e.}, 
$(y)=J_1\fA$ for an (integral) ideal $\fA\subseteq\bZ[\la]$ of $\bZ[\la]$. Note that 
$\fA$ is not principal ({\it i.e.}, $\fA\ne x\bZ[\la]$ for any $x\in K$), since the class of $J_1$ is non-trivial. Analogously, \eqref{eq:contr} implies $5^tJ_1\subseteq (y)$ and hence $5^tJ_1=(y)\fA'$ 
for an (integral) ideal $\fA'\subseteq\bZ[\la]$ of $\bZ[\la]$. Combining the two equalities, we get
$$
5^tJ_1=(y)\fA'=J_1\fA\fA'.
$$
Cancelling $J_1$, this implies $(5^t)=\fA\fA'$. Using \cite{sage}, we can check that 
all the prime ideal divisors of the ideal $(5)$ are principal, hence $\fA$ is principal and so is $J_1$, which is a contradiction. This shows $[I_{\caR}(A,\la)]\ne[I_{\caR}(B,\la)]$. Nonetheless, we show next that $G_A\cong G_B$.

\sbr

By \cite{sage}, $(5)=\fp_1^2\fp_2^2$, where $\fp_1,\fp_2$ are prime ideals of 
$\bZ[\la]$, $\fp_1=(\la)$, and there exists $g\in \Gal(K/\bQ)$ of order 
$2$ such that 
$g(\fp_i)=\fp_i$, $i=1,2$. In the notation of Theorem \ref{th:prelim0}, 
$X_{A,\fp_1}=\Span_{K}({\bf u},g({\bf u}))$, 
$X_{B,\fp_1}=\Span_{K}({\bf v},g({\bf v}))$. We look for $f_1,f_2\in K$ such that
$f_1{\bf v}+f_2g({\bf v})\in\caR[\la]$. Using the action of $g$, the condition is equivalent to
the existence of $T\in\GL_4(\caR)$ with $T(X_{A,\fp_1})=X_{B,\fp_1}$, namely,
$f_1{\bf v}+f_2g({\bf v})=T({\bf u})$. Note that any element in $K$ can be written as $\bQ$-linear combination of $1,\la,\la^2,\la^3$, since $K=\bQ(\la)$ of degree $4$ over $\bQ$.  In other words, for any $f_1,f_2\in K$ there is 
$L\in\GL_4(\bQ)$ such that $f_1{\bf v}+f_2g({\bf v})=L({\bf u})$. The goal is to find 
$f_1,f_2\in K$ so that both $L$, $L^{-1}$ have coefficients in $\caR$, {\it i.e.}, the denominators of coefficients of both $L$, $L^{-1}$ are powers of $5$. It turns out that such $f_1,f_2$ exist, namely,
\begin{eqnarray*}
f_1&=&\frac{39}{350}\la^3-\frac{29}{175}\la^2+\frac{739}{350}\la-\frac{5}{14},\\
f_2&=&\frac{61}{350}\la^3-\frac{46}{175}\la^2+\frac{1261}{350}\la-\frac{27}{14},
\end{eqnarray*}
and $f_1{\bf v}+f_2g({\bf v})=T({\bf u})$ with
$$
T=\left(\begin{matrix}
-21 & 40 & -3 & 2 \\
-\frac{72}{5} & \frac{141}{5} & -\frac{11}{5} & \frac{7}{5} \\
0 & 1 & 0 & 0\\
-20 & 40 & -3 & 2
\end{matrix}
\right),\quad \det T=-\frac{1}{5},\quad T\in\GL_4(\caR).
$$
We use Theorem \ref{th:prelim0} to show that $T$ is an isomorphism from $G_A$ to $G_B$, {\it i.e.}, $T(G_A)=G_B$. 
First, we find characteristics of $G_A$, $G_B$. We apply the process described in the proof of  Lemma \ref{l:vect} to vectors ${\bf u}, g({\bf u})$. We have
$$
{\bf u}=\left(\begin{matrix}
1 & \la & \la^2 & \la^3
\end{matrix}\right)^t,\quad 
g({\bf u})=\left(\begin{matrix}
1 & g(\la) & g(\la^2) & g(\la^3)
\end{matrix}\right)^t,
$$
where 
\begin{eqnarray*}
g(\la)&=&-4\la^3+6\la^2-81\la+40,\\
g(\la^2)&=&-4\la^3+5\la^2-80\la+20,\\
g(\la^3)&=&75\la^3-114\la^2+1520\la-770.
\end{eqnarray*}
Applying column operations on $\left(\begin{matrix} {\bf u} & g({\bf u}) \end{matrix}\right)$
corresponding to multiplications by matrices from $\GL_4(\bZ_5)$, we arrive at
$$
\left(\begin{matrix}
1 & 0 & -\de & -2\de+10\\
0 & 1 & 2\de+40 & 3\de+40
\end{matrix}\right)^t,\quad \de=-2\la^3+3\la^2-40\la+20.
$$
Therefore, 
$$
M(A;{\bf e}_1,\ldots,{\bf e}_4)=\{\al_{513}(A),\al_{514}(A),\al_{523}(A), \al_{524}(A)\},
$$ 
where
\begin{alignat}{2}
\al_{513}(A) & = -\de, & \quad  \al_{514}(A)&=-2\de+10, \nonumber \\
\al_{523}(A)& = 2\de+40, & \quad  \al_{524}(A)&=3\de+40.\nonumber 
\end{alignat}
Note that $K_{\fp_1}$ is an extension of $\bQ_5$ of degree $2$
and $\Gal(K_{\fp_1}/\bQ_5)$ is generated by $g$. Since $\de\in\bZ[\la]$, under an embedding 
$K\hrar K_{\fp_1}$, $\de$ becomes an element of the ring of integers of $K_{\fp_1}$. Since
 $\de=\la\cdot g(\la)$, $\de$ is an integral element of $\bQ_5$ and therefore, $\de\in\bZ_5$.
 Therefore, all the elements $\al_{5ij}(A)$ in $M(A;{\bf e}_1,\ldots,{\bf e}_4)$ belong to $\bZ_5$. To find a characteristic of $G_B$, we repeat the above process for vectors
 ${\bf v}, g({\bf v})$. We arrive at
$$
\left(\begin{matrix}
1 & 0 & \frac{1}{7}(1-4\de) & \frac{1}{7}(\de+5) \\
0 & 1 & \de & 0 
\end{matrix}\right)^t,\quad \de=-2\la^3+3\la^2-40\la+20,
$$
and
$$
M(B;{\bf e}_1,\ldots,{\bf e}_4)=\{\al_{513}(B),\al_{514}(B),\al_{523}(B), \al_{524}(B)\},
$$ 
where
\begin{alignat}{2}
\al_{513}(B)& = \frac{1}{7}(1-4\de), & \quad  \al_{514}(B)&=\frac{1}{7}(\de+5), \nonumber \\
\al_{523}(B)& = \de, & \quad  \al_{524}(B)&=0.\nonumber 
\end{alignat}
Note that all $\al_{5ij}(B)\in\bZ_5$. 
We can now check the condition $(A,B,5)$ for $T$ in Theorem \ref{th:prelim0}.
It holds, because $\al(B)_{523}=\de$, 
$\al(B)_{524}=0$ are both divisible by $5$ in $\bZ_5$ (by the choice of $\fp_1$, $\la$ is divisible by $\fp_1$ in $\cO_{\fp_1}$).
Since $T^{-1}$ has integer coefficients, the condition $(B,A,5)$ holds automatically. In Theorem \ref{th:prelim0}, the conditions 
\begin{alignat}{2}
\cP(A)&=\cP(B)=\{5\}, & \quad  & \caR=\caR(A)=\caR(B), \nonumber \\
\cP'(A)&=\cP'(B)=\{5\}, & \quad  & t_5(A)=t_5(B)=2 \nonumber 
\end{alignat}
hold automatically, since $A$, $B$ share the same eigenvalues. Also, $\Gal(K/\bQ)$ acts transitively on the prime ideals $\fp_1,\fp_2$ above $5$, so there exists 
$g'\in\Gal(K/\bQ)$ such that $g'(\fp_1)=\fp_2$. 
By above, $T(X_{A,\fp_1})=X_{B,\fp_1}$, $T\in\GL_4(\caR)$, and applying $g'$, we get
$T(X_{A,\fp_2})=X_{B,\fp_2}$. By Theorem \ref{th:prelim0}, $G_A\cong G_B$, but
$[I_{\caR}(A,\la)]\ne [I_{\caR}(B,\la)]$, even though the characteristic polynomials of $A$, $B$ are irreducible over $\bQ$.

\begin{example}\label{ex:6} 
The motivation behind this example is the following question. Assume the characteristic polynomial of $A$ is irreducible and $G_A\cong G_B$. Is 
necessarily the characteristic polynomial of $B$ also irreducible? This is true for $n=2$ (see \cite[Remark 4.2]{s}) and it turns out that this is not true for an arbitrary $n$. In our example, $n=4$, 
$A,C\in\M_4(\bZ)$ have the same irreducible characteristic polynomial $h(t)=t^4+t^2+9$, and $G_A=G_C$. Let $B=C^2$. Then the minimal polynomial of $B$ is
$t^2+t+9$, so that the characteristic polynomial of $B$ is $(t^2+t+9)^2$, not irreducible. However, 
$G_B=G_C=G_A$. More precisely,
$$
A=\left(\begin{matrix}
0 & 1 & 0 & 0 \\
0 & 0 & 1 & 0 \\
0 & 0 & 0 & 1 \\
-9 & 0 & -1 & 0
\end{matrix}
\right),\quad
C=\left(\begin{matrix}
0 & 1 & -1 & 0 \\
9 & 0 & 2 & 1 \\
9 & 0 & 1 & 1 \\
-18 & -9 & 7 & -1
\end{matrix}
\right),
$$
where $\det A=\det C=9$, $\cP=\cP'=\{3\}$, and $t_3=2$. By Hensel's lemma, there exists
a root $\la\in\overline{\bQ}$ of $h$ such that $\la\in\bZ_3^{\times}$ under 
$\bQ(\la)\hrar \bQ(\la)_{\fp}$, where $\fp$ is a prime ideal of the ring of integers of
$\bQ(\la)$  above $3$.
 One can show that
$$
G_A=G_C=<{\bf e}_1,\ldots,{\bf e}_4,3^{-\infty}({\bf e}_1+\la^2{\bf e}_3),3^{-\infty}({\bf e}_2+\la^2{\bf e}_4)>.
$$
(For example, we can apply the process described in the proof of  Lemma \ref{l:vect} to eigenvectors
$$
{\bf u}_i=\left(\begin{matrix}
1 & \pm\la & \la^2 & \pm\la^3
\end{matrix}
\right)^t,\quad
{\bf v}_i=\left(\begin{matrix}
1 & \pm\la+\la^2 & \la^2 & \pm\la^3-\la^2-9
\end{matrix}
\right)^t,\,\,i=1,2,
$$
of $A$, $C$, respectively, corresponding to $\pm\la$.) Thus, $G_B=G_C=G_A$, the characteristic polynomial of $A$ is irreducible, and the characteristic polynomial of 
$B$ is not irreducible.
%Thus, $\operatorname{Id}$ is an isomorphism between $G_A$ and $G_B$. However, $A$, $B$ do not share the same eigenvectors, since $A\ne B$. In this example, we have $SAS^{-1}=B$, $S\in\SL_4(\bZ)$,
%$$
%S=\left(\begin{matrix}
%1 & 0 & 0 & 0 \\
%0 & 1 & 1 & 0 \\
%0 & 0 & 1 & 0 \\
%-9 & 0 & -1 & 1
%\end{matrix}
%\right),
%$$
%so that $S(G_A)=G_B$ and $S$ is another isomorphism from $G_A$ to $G_B$.
\end{example}

\end{example}

\section{applications}
%%%%%%%%%%%%%%%%%%%%%%%%%%%%
%%%%%%%%%%%%%%%%%%%%%%%%%%%%
%%%%%%%%%%%%%%%%%%%%%%%%%%%%
%%
%%       from the previous paper
%%
%%%%%%%%%%%%%%%%%%%%%%%%%%%%
%%%%%%%%%%%%%%%%%%%%%%%%%%%%
%%%%%%%%%%%%%%%%%%%%%%%%%%%%

\subsection{$\bZ^n$-odometers} In this section we generalize our results in \cite{s} on application of groups $G_A$ to 
$\bZ^2$-odometers to the $n$-dimensional case. By definition, a
$\bZ^n$-odometer is a dynamical system 
consisting of a topological space $X$ and an action of the group $\bZ^n$ on $X$ (by homeomorphisms). There is a way to construct a 
$\bZ^n$-odometer out of a  subgroup $H$ of $\bQ^n$ that contains $\bZ^n$ \cite[p.\ 914]{gps}.  
Namely, the associated odometer $Y_H$ is the Pontryagin dual of the quotient $H/\bZ^n$, {\it i.e.,} $Y_H=\widehat{H/\bZ^n}$. The action of $\bZ^n$ on $Y_H$ is given as follows.  
Let $\rho$ denote the embedding
$$
\rho:H/\bZ^n\hrar \bQ^n/\bZ^n\hrar \bT^n,\quad \bT^n=\bR^n/\bZ^n.
$$
Identifying Pontryagin dual $\widehat{\bT^n}$ of $\bT^n$ with $\bZ^n$, we have the induced map
$$
\widehat\rho:\bZ^n \rar Y_H=\widehat{H/\bZ^n}.
$$
The action of $\bZ^n$ on $Y_H$ is given by $\widehat\rho$. 
Let $A\in\M_n(\bZ)$ be non-singular. Applying the process to the group $H=G_A$, we get the associated $\bZ^n$-odometer $Y_{G_A}$. For simplicity, we denote $Y_{G_A}$ by $Y_{A}$. 

%\sbr
%
%Another way to obtain a $\bZ^n$-odometer is to consider a decreasing sequence of finite-index subgroups of $\bZ^n$
%$$
%G=\bZ^n\supseteq G_1\supseteq G_2\supseteq\cdots
%$$
%and the natural maps $\pi_i:G/G_{i+1}\rar G/G_{i}$, $i\in\bN$. The associated $\bZ^n$-odometer is  the inverse limit 
%\bbe\label{eq:xg}
%X=\lim_{\longleftarrow} \left(G/G_{i}\right)
%\ee
%together with the natural action of $\bZ^n$. 
% For the sequence 
%$$
%G_i=\left.\left\{A^{i}{\bf x}\,\right\vert\, {\bf x}\in\bZ^n\right\},\quad i\in\bN,
%$$
%denote by $X_A$ the corresponding odometer. Using duality, one can prove that $X_A$ and $Y_{{A^t}}$ are conjugate 
%$\bZ^n$-odometers \cite[Theorem 2.6]{gps}.

%%%%%%%%%%%%%%%%%%%%%%%%%%%%%%%%%%%%
%%%%%%%%%%%%%%%%%%%%%%%%%%%%%%%%%%%%
%%%%%%%%%%%%%%%%%%%%%%%%%%%%%%%%%%%%
\sbr

In the next lemma we analyze when $G_A$ is dense in $\bQ^n$. The result generalizes the case $n=2$ \cite[Lemma 8.4]{s}. Let $A\in\M_n(\bZ)$ be non-singular
and let $h_A\in\bZ[t]$ be the characteristic polynomial of $A$. 
Let $h_A=h_1h_2\cdots h_s$,
where $h_1,\ldots,h_s\in\bZ[t]$ are irreducible of degrees $n_1,\ldots,n_s$, respectively.
By
Theorem \ref{th:trig} below, there exists $S\in\GL_n(\bZ)$ such that
\bbe\label{eq:kkkk1}
SAS^{-1}=\left(\begin{matrix}
A_{1} & *  & \cdots & * \\
0 & A_{2}  & \cdots & * \\
\vdots & \vdots &  &\vdots \\
0 & 0  & \cdots & A_{s}
\end{matrix}
\right),
\ee
where each $A_i\in\M_{n_i}(\bZ)$ has characteristic polynomial $h_i$, $i\in\{1,2,\ldots,s\}$.

\begin{lem}\label{l:dense}
$G_A$ is dense in $\bQ^n$ if and only if $A_i\not\in\GL_{n_i}(\bZ)$ for all $i\in\{1,2,\ldots,s\}$. Equivalently, $G_A$ is dense in $\bQ^n$ if and only if $\det A_i\ne\pm 1$ for all $i\in\{1,2,\ldots,s\}$ if and only if $h_i(0)\ne\pm 1$ for all $i\in\{1,2,\ldots,s\}$.
\end{lem}

\begin{proof}
As in the proof of Lemma 8.4 in \cite{s}, $G_A$ is dense in $\bQ^n$ if and 
only if 
\bbe\label{eq:y}
A^{-i}{\bf y}\in\bZ^n\text{ for any }i\in\bN,\quad {\bf y}\in\bZ^n,
\ee
implies ${\bf y}={\bf 0}$. We first show that if there exists $A_i\in\GL_{n_i}(\bZ)$, then $G_A$ is not dense. Indeed, without loss of generality, we can assume that $A$ itself has the block upper-triangular form \eqref{eq:kkkk1} and that $A_1\in\GL_{n_1}(\bZ)$. Then for any ${\bf y}_0\in\bZ^{n_1}$ and $i\in\bN$, $A_1^{-i}{\bf y}_0\in\bZ^{n_1}$, so that there exists non-zero ${\bf y}=\left(\begin{matrix}
{\bf y}_0 & {\bf 0}
\end{matrix}
\right)^t\in\bZ^n$ satisfying \eqref{eq:y}, and $G_A$ is not dense. 

\sbr

We are now left to show that if $G_A$ is not dense, then there exists $A_i\in\GL_{n_i}(\bZ)$.
We first consider the case when $h_A$ is irreducible. Assume $G_A$ is not dense, hence there exists ${\bf y}\ne{\bf 0}$ satisfying \eqref{eq:y}. 
Note that $A$ is diagonalizable with eigenvectors ${\bf u}_1,\ldots,{\bf u}_n\in\bC^n$, linearly independent over 
$\bC$, corresponding to eigenvectors $\la_1,\ldots,\la_n\in\bC$, respectively. 
Let $M=\left(\begin{matrix}
{\bf u}_1 & \ldots & {\bf u}_n
\end{matrix}
\right)\in\GL_n(\bC)$.
Let $K$ be a finite Galois extension of $\bQ$ that contains all the eigenvalues of $A$ and let 
$\cO_K$ denote its ring of integers. Without loss of generality, we can assume that $M\in\M_n(\cO_K)$, so that $\det M\in\cO_K-\{0\}$. Let ${\bf y}\in\bZ^n$ satisfy \eqref{eq:y}, ${\bf y}=\sum_{j=1}^nc_j{\bf u}_j$, $c_1,\ldots,c_n\in K$, not all are zeroes. Then \eqref{eq:y} implies
$$
A^{-i}{\bf y}=M\cdot\left(\begin{matrix}
c_{1}\la_1^{-i} & c_{2}\la_2^{-i}  & \ldots & c_{n}\la_n^{-i}\end{matrix}
\right)^t\in\bZ^n.
$$
Thus, multiplying the last formula (on the left)  by the adjoint matrix $\tilde M\in\M_n(\cO_K)$ of $M$, we have
\bbe\label{eq:oko}
\det Mc_j\la_j^{-i}\in\cO_K\text{ for any }i\in\bN\text{ and }j\in\{1,2,\ldots,s\}.
\ee
Since there exists $c_k\ne 0$ for some $k\in\{1,\ldots,n\}$ and $\det M\ne 0$, we have $\la_k\in\cO_K^{\times}$, {\it i.e.}, $\la_k$ is a unit in $\cO_K$. Indeed, otherwise there exists a prime ideal $\fp$ of $\cO_K$ dividing $\la_k$. Then, writing, $c_k=\ga_k/\de_k$, $\ga_k,\de_k\in\cO_K-\{0\}$, from \eqref{eq:oko} for $j=k$ we get that non-zero $\det M\ga_k\in\cO_K$ (which does not depend on $i$) is divisible by arbitrary powers $\fp^i$, $i\in\bN$, which is impossible. Since $h_A$ is irreducible by assumption, $\Gal(\overline{\bQ}/\bQ)$ acts transitively on the set of eigenvalues of $A$. Thus, since there is one eigenvalue $\la_k\in\cO_K^{\times}$, all the eigenvalues of $A$ are units in $\cO_K$ and their product $\la_1\la_2\cdots \la_n=\det A$ is a unit in $\bZ$, {\it i.e.}, $\det A=\pm 1$ and $A\in\GL_n(\bZ)$.

\sbr

We now assume that $h_A$ is not irreducible. We need to show that if $G_A$ is not dense, then there exists $A_i\in\GL_{n_i}(\bZ)$. Equivalently, if all 
$A_i\not\in\GL_{n_i}(\bZ)$, then $G_A$ is dense. Assume all 
$A_i\not\in\GL_{n_i}(\bZ)$.
We prove that this implies that $G_A$ is dense by induction on the number of irreducible components of $h_A$; the base of the induction (the case of one irreducible component) is considered in the preceding paragraph. Let $h_A=h_1h_2$, where $h_1,h_2\in\bZ[t]$ are monic polynomials of degrees $n_1,n_2\in\bN$, respectively. By
Theorem \ref{th:trig} below, there exists $T\in\GL_n(\bZ)$ such that
\bbe\label{eq:kkkk}
TAT^{-1}=\left(\begin{matrix}
A_{1} & *   \\
0 & A_{2} 
\end{matrix}
\right),
\ee
where each $A_i\in\M_{n_i}(\bZ)$ has characteristic polynomial $h_i$, $i=1,2$.
Without loss of generality, we can assume that $A$ itself has the block triangular form \eqref{eq:kkkk}. Clearly,
$C\not\in\GL(\bZ)$ for any ``irreducible" block $C$ of $A_1$, $A_2$. Then, by induction, $G_{A_i}$ is dense in $\bQ^{n_i}$, $i=1,2$. Namely, if ${\bf y}\in\bZ^n$ satisfies \eqref{eq:y} and ${\bf y}=\left(\begin{matrix}
{\bf y}_1 & {\bf y}_2
\end{matrix}
\right)^t$, ${\bf y}_i\in\bZ^{n_i}$, $i=1,2$, then $A_2^{-i}{\bf y}_2\in\bZ^{n_2}$
for all $i\in\bN$, and hence ${\bf y}_2={\bf 0}$ by induction. Then, \eqref{eq:y} implies $A_1^{-i}{\bf y}_1\in\bZ^{n_1}$
for all $i\in\bN$, and hence ${\bf y}_1={\bf 0}$ by induction as well. Thus, 
${\bf y}={\bf 0}$ and $G_{A}$ is dense in $\bQ^{n}$.

\sbr

The other two equivalent formulations follow from the facts that $A\in\M_n(\bZ)$ belongs to
$\GL_n(\bZ)$ if and only if $\det A=\pm 1$ and if $h\in\bZ[t]$ is the characteristic polynomial of $A$, 
then $\det A=(-1)^nh(0)$.
\end{proof}

\begin{lem}\label{l:orbit}
Let $A,B\in\M_n(\bZ)$ be non-singular such that 
$G_A$ $($resp., $G_B$$)$ is dense in $\bQ^n$ $($see Lemma $\ref{l:dense})$. Then $\bZ^n$-actions $Y_A$, $Y_B$ are orbit equivalent if and only if $\det A,\det B$ have the same prime divisors.
\end{lem}
\begin{proof}
Follows from \cite[Theorem 1.5]{gps} and \cite[Lemma 8.5]{s}.
\end{proof}

%\begin{rem}
In \cite[Theorem 1.5]{gps}, the authors give a characterization of various equivalences of $\bZ^2$-odometers $Y_H$ in terms of the corresponding groups $H$. In our subsequent paper, we extend their results to the $n$-dimensional case of $\bZ^n$-odometers and apply them for odometers of the form $Y_A$ defined by non-singular matrices $A\in\M_n(\bZ)$.
%\end{rem}

%\begin{rem}
%Lemma \ref{l:orbit} in the case $n=2$ is proved in \cite[Lemma 8.6]{s}.
%\end{rem}

\section{Similarity to a block-triangular matrix over PID}\label{ap:trig}
In this section we give a proof of the fact that a matrix $A$ over a principal ideal domain $R$ with field of fractions of characteristic zero is similar over $R$ to a block-triangular matrix. 
%with blocks corresponding to irreducible factors of the characteristic polynomial of $A$. 
This is proved in \cite[p. 50, Thm. III.12]{n} for $R=\bZ$ and the same proof works for a general principal ideal domain (PID) with field of fractions of characteristic zero. In particular, when $R=\bZ_p$, the case of our interest. We repeat the 
proof here with a slight modification, which is useful in  calculating examples.
\begin{thm}\label{th:trig}
Let $R$ be a PID with field of fractions of characteristic zero. For any $A\in\M_n(R)$ there exists $S\in\GL_n(R)$ such that
$$
SAS^{-1}=\left(\begin{matrix}
A_{11} & *  & \cdots & * \\
0 & A_{22}  & \cdots & * \\
\vdots & \vdots &  &\vdots \\
0 & 0  & \cdots & A_{tt}
\end{matrix}
\right),
$$
where each $A_{ii}$ is a square matrix with irreducible characteristic polynomial, $i\in\{1,2,\ldots,t\}$, $1\leq t\leq n$.
\end{thm}
\begin{proof}
Let $F$ denote the field of fractions of $R$ and let $h_A\in R[t]$ denote the characteristic polynomial of $A$. If $h_A$ is irreducible, there is nothing to prove. Assume $h_A$ is not irreducible, {\it i.e.}, $h_A=h_1h_2$, where $h_1,h_2\in R[t]$ are monic, and $h_1$  
is irreducible of degree $k$, $1\leq k< n$. Let $\overline{F}$ denote a fixed algebraic closure of $F$, let $\al\in\overline{F}$ be a root of $h_1$, and let 
$L=F(\al)$. Then $L$ is a finite separable extension of $F$ of degree $k$. It is well-known that $L$ is the field of fractions of $R[\al]$. 
%and let $\cO$ denote the integral closure of $R$ in $L$. 
%It is known that $\cO$ is a free $R$-module of rank $k$ and hence there exists a basis $\om_1,\ldots,\om_k\in\cO$ of $\cO$ over $R$. 
Let ${\bf u}\in (\overline{F})^n$ be an eigenvector of $A$ corresponding to $\al$. Without loss of generality, we can assume that ${\bf u}\in R[\al]^n$. Then
$$
{\bf u}=C{\bf\om},\quad {\bf\om}=\left(\begin{matrix} 1 &\al & \ldots & \al^
{k-1} \end{matrix}\right)^t
$$
for some $C\in\M_{n\times k}(R)$. Also, there exists $B\in\M_{k}(R)$ such that $\al{\bf\om}=B{\bf \om}$. Then
$$
A{\bf u}=AC{\bf \om}=\al C{\bf\om}=CB{\bf \om}
$$
and hence $AC=CB$, since
entries of $AC-CB$ belong to $R$ and 
$1,\al,\ldots,\al^{k-1}$ is a basis of $L$ over $F$. Since $R$ is a PID, matrix $C$ has a Smith normal form, {\it i.e.}, there exist $\la_1,\ldots,\la_r\in R-\{0\}$, $U\in\GL_n(R)$, and $V\in\GL_k(R)$ such that
$$
C=UTV,\quad T=\left(\begin{matrix}
\La & 0 \\
0 & 0
\end{matrix}
\right),
$$
where $T\in\M_{n\times k}(R)$, $\La=\diag(\la_1,\ldots,\la_r)$ is a non-singular diagonal matrix, and $r\leq k$. We write
$$
U^{-1}AU=\left(\begin{matrix}
A_1 & A_2 \\
A_3 & A_4
\end{matrix}
\right),
$$
where $A_1\in\M_r(R)$, and $A_2,A_3,A_4$ are matrices over $R$ of appropriate sizes. It follows from $AC=CB$ that
\bbe\label{eq:last}
\left(\begin{matrix}
A_1 & A_2 \\
A_3 & A_4
\end{matrix}
\right)\left(\begin{matrix}
\La & 0 \\
0 & 0
\end{matrix}
\right)V=
\left(\begin{matrix}
\La & 0 \\
0 & 0
\end{matrix}
\right)VB.
\ee
Thus, $A_3\La=0$ and since $\La$ is non-singular, we have $A_3=0$. We now show that $\al$ is an eigenvalue of $A_1$ and hence $k=r$. Indeed, multiplying \eqref{eq:last} by ${\bf \om}$ on the right, we get
\bbe\label{eq:lkl}
\left(\begin{matrix}
A_1 & A_2 \\
A_3 & A_4
\end{matrix}
\right)\left(\begin{matrix}
\La & 0 \\
0 & 0
\end{matrix}
\right)V{\bf \om}=
\left(\begin{matrix}
\La & 0 \\
0 & 0
\end{matrix}
\right)VB{\bf \om}=
\al\left(\begin{matrix}
\La & 0 \\
0 & 0
\end{matrix}
\right)V{\bf \om},
\ee
since $B{\bf \om}=\al{\bf\om}$.
Let ${\bf v}\in\M_{r\times 1}(L)$ denote the first $r$ entries of $V{\bf\om}\in\M_{k\times 1}(L)$ and let ${\bf w}=\La{\bf v}$. Note that ${\bf v}$ is non-zero, since ${\bf \om}$ is a basis and $V$ is non-singular. Also, ${\bf w}$ is non-zero, since $\La=\diag(\la_1,\ldots,\la_r)$ is non-singular.
Then \eqref{eq:lkl} implies 
$$
A_1{\bf w}=\al{\bf w}.
$$
Since ${\bf w}$ is non-zero, $\al$ is an eigenvalue of $A_1$. Hence, $k=r$, 
$h_1$ is the characteristic polynomial of $A_1$, and $h_2$ is the characteristic polynomial of $A_4$. Applying the induction process on $n$, the statement of the theorem holds for $A_4\in\M_{n-k}(R)$ and therefore, holds for $A$. 
\end{proof}

\end{document}